\documentclass[a4paper,12pt,reqno]{amsart}
\usepackage{amsmath}
\usepackage{amsthm,enumerate}

\usepackage{graphicx}
\usepackage{amssymb}

\usepackage{appendix}
\usepackage{fancyhdr}

\usepackage{calc}
\usepackage{url}

\usepackage[text={6in,9in},centering]{geometry}

\usepackage{srcltx, inputenc}

\usepackage[T1]{fontenc}

\usepackage[usenames,dvipsnames]{color}

\makeatletter
\def\cleardoublepage{\clearpage\if@twoside \ifodd\c@page\else%
         \hbox{}%
     \thispagestyle{empty}
     \newpage%
     \if@twocolumn\hbox{}\newpage\fi\fi\fi}
\makeatother

\let\cleardoublepage\clearpage

\newtheorem{thm}{Theorem}[section]

\newtheorem{lem}[thm]{Lemma}

\newtheorem{den}[thm]{Definition}
\newtheorem{oss}[thm]{Remark}
\numberwithin{equation}{section}

\newcommand{\ren}{\mathbb{R}^n}

\usepackage{xcolor}

\begin{document}

\title[The PME on manifolds with negative, superquadratic curvature]{The porous medium equation\\ [6pt] on Riemannian manifolds with negative curvature: \\ [6pt] The superquadratic case \\}

\author {Gabriele Grillo, Matteo Muratori}

\author{Juan Luis V\'azquez}

\address {Gabriele Grillo: Dipartimento di Matematica, Politecnico di Milano, Piaz\-za Leo\-nar\-do da Vinci 32, 20133 Milano, Italy}
\email{gabriele.grillo@polimi.it}

\address{Matteo Muratori: Dipartimento di Matematica, Politecnico di Milano, Piaz\-za Leo\-nar\-do da Vinci 32, 20133 Milano, Italy}
\email{matteo.muratori@polimi.it}

\address{Juan Luis V\'azquez: Departamento de Matem\'aticas, Universidad Aut\'onoma de Ma\-drid, 28049 Madrid, Spain}
\email{juanluis.vazquez@uam.es}


%
%
%
%


\begin{abstract} We study the long-time behaviour of nonnegative solutions of the Porous Medium Equation posed on Cartan-Hadamard  manifolds having very large negative curvature, more precisely when the sectional or Ricci curvatures diverge at infinity more than quadratically in terms of the geodesic distance to the pole. We find an unexpected separate-variable behaviour that reminds one of Dirichlet problems on bounded Euclidean domains.
As a crucial step, we prove existence of solutions to a related sublinear elliptic problem, a result of independent interest. Uniqueness of solutions vanishing at infinity is also shown, along with comparison principles, both in the parabolic and in the elliptic case.

Our results complete previous analyses of the Porous Medium Equation flow on negatively curved Riemannian manifolds, which were carried out first for the hyperbolic space and then for general Cartan-Hadamard manifolds with a negative curvature having at most quadratic growth. We point out that no similar analysis  seems to exist for the linear heat flow.

We also translate such results into some weighted Porous Medium Equations in the Euclidean space having special weights.
\end{abstract}

\maketitle

\section{Introduction}
We study the porous medium equation (PME for short)
\begin{equation}\label{pme}
\begin{cases}
u_t=\Delta u^m  & \text{in } M \times \mathbb{R}^+ \, , \\
u(\cdot,0)=u_0  & \text{in } M  \, ,
\end{cases}
\end{equation}
where $m>1$, $\Delta=\Delta_g$ is the Laplace-Beltrami operator on a complete Riemannian manifold $(M,g)$ of dimension $n\ge2$ and the initial datum $u_0$ is assumed to be nonnegative, bounded and, in some of our results, compactly supported\footnote{These conditions can be relaxed, see below, but they allow to simplify the presentation.}. The main assumption on $M$ is that it is a
\it Cartan-Hadamard \rm manifold, namely that it is complete, simply connected and has everywhere \it nonpositive \rm sectional curvature.  It is well known that a minimal solution to \eqref{pme} always exists for bounded data, see Section \ref{mainres} for some detail.

The study of the above problem has been initiated in \cite{V}, where the accurate long-time behaviour of the PME was established on the hyperbolic space $\mathbb{H}^n$  by means of sharp space-time asymptotic estimates for solutions corresponding to compactly supported data. The results on the size and shape of the solutions for large times, as well as the size of their support, are completely different from the well-known results on the Euclidean space, see \cite{V07}. They were then considerably extended in \cite{GMV}, where corresponding asymptotic results were shown to hold for general Cartan-Hadamard manifolds under only appropriate (sectional, radial) \it curvature conditions\rm. In fact, let $ \mathrm{K}_\omega(x)$ denote the  \emph{sectional curvature} with respect to any $2$-dimensional tangent subspace $ \omega $ at $x$ that contains the radial direction w.r.t.\ a pole $o$. The bound we assume will be of the form \begin{equation}\label{upper-intro}
\mathrm{K}_\omega(x) \le -Q \, r^{2 \mu} \quad \forall x \in M \setminus B_R(o) \, ,
\end{equation}
where $ Q>0 $, $ B_R(o)$ is the geodesic ball of radius $ R>0 $ centered at $o$, and $r=r(x)$ represents the geodesic distance $\operatorname{d}(x,o)$ between $o$ and $x$, or similar lower bounds on the radial Ricci curvature. In \cite{GMV} the range of exponents $-\infty <\mu \le 1$ was covered, and a precise classification was found:  the exponents $\mu<-1$ correspond to quasi-Euclidean behaviour, namely
\begin{equation*}
\|u(t)\|_{L^\infty(M)} \sim t^{-\frac{1}{m-1+2/n}}
\end{equation*}
at large $t$. On the contrary, for $-1 <\mu<1$ we speak of quasi-hyperbolic behaviour, which is characterized, still for large times, by the decay rate
\begin{equation*}
\|u(t)\|_{L^\infty(M)}^{m-1} \sim \frac{ (\log t)^{\frac{1-\mu}{1+\mu}} }{t} \,,
\end{equation*}
as well as the support propagation rate
\begin{equation}\label{radius}
R(t) \sim \kappa \, (\log t)^{\frac{1}{1+\mu}} \, , \quad \kappa  > 0 \,,
\end{equation}
where $R(t)$ indicates the location of  any point of the free boundary measured in geodesic distance from $o$. The delicate transition at $\mu=-1$ is also considered in detail in \cite{GMV}, and it corresponds to Euclidean-type behaviours with a ``fractional'' dimension that depends on $ n $ and $Q$.

The following more precise information is relevant. It happens that the behaviour of all solutions we have considered so far mimics the fundamental solution of the corresponding problem, i.e.~the solution with initial datum a Dirac mass located at a fixed point that we take as origin of coordinates and origin of the geodesic distance. As shown in \cite{V} for the hyperbolic case, such a solution is not self-similar and this is usually a source of difficulties, but in fact it is approximately self-similar for all large times. In \cite{GMV} such approximation is expressed for $\mu\in (-1,1) $ by means of sub- and super-solutions of the form
\begin{equation}\label{barrier1}
U(x,t)^{m-1}\sim \frac{C}{t}\left[ (\mu-1) \left( \gamma\,[\log t]^{\frac{1-\mu}{1+\mu}}-r^{1-\mu} \right) \right]_+
\end{equation}
for suitable constants $C,\gamma>0$ involving the initial datum.

Finally, the analysis ends at the critical case $\mu =1$, which exhibits special asymptotic features, typical of a borderline situation. In particular, the $L^\infty$ asymptotic estimate is shown to be
\begin{equation*}
\|u(t)\|_{L^\infty(M)}^{m-1} \sim \frac{\log \log t}{t}\,.
\end{equation*}

\medskip

\noindent {\bf New problem on Cartan-Hadamard manifolds and main results.} It was pointed out in \cite{GMV} that if sectional curvatures are very negative in the sense that $\mu>1$, i.e.~if they are allowed to grow faster than quadratically at infinity, a new type of  asymptotic pattern should arise. Investigating the properties of flows on manifolds with such a highly divergent curvature seems to go beyond the scope of the existing literature. We stress that, to the best of our knowledge, there is no existing discussion of the \it linear \rm case, namely of heat kernel behaviour, under the present curvature conditions. See \cite[Section 10]{GMV} for further details.

\noindent $\bullet$ The purpose of this paper is to address that open problem and find the missing behaviour. More precisely, we study the asymptotic behaviour of solutions to \eqref{pme} under the curvature condition \eqref{upper-intro} with $\mu>1$.  For the validity of some of our results, a matching lower bound on the radial Ricci curvature will be also assumed. It turns out that the long-time asymptotics is significantly different from the case $\mu\le1$. Indeed,  the evolution is now controlled by a \it separable \rm solution $ \mathcal{U} $ of the form
\[
\mathcal{U}(x,t)=\frac{v(x)}{t^{\frac1{m-1}}} \quad \forall (x,t) \in M \times \mathbb{R}^+ \, ,
\]
in which the spatial part is a positive solution of the \it nonlinear elliptic problem \rm
\begin{equation}\label{elliptic}
-\Delta v^m=\frac{1}{m-1} \, v \quad \text{in } M \, .
\end{equation}
In fact, a crucial part of the present work consists in showing that problem \eqref{elliptic} does have a \emph{minimal}, \emph{positive} regular solution, which tends to zero at infinity with a certain explicit power rate, see Theorem \ref{thm-ell-ex} below.

As just pointed out, the key assumption for such result to be true is that the curvature condition \eqref{upper-intro} holds with $\mu>1$: indeed, in this case, problem \eqref{elliptic} is reminiscent of its analogue for the PME posed on a \it bounded domain \rm in the \it Euclidean space\rm, or a certain \it weighted \rm nonlinear elliptic problem in the whole ${\mathbb R}^n$ as studied in   \cite{BK}. Actually, the latter connection is more than a mere analogy, since a suitable change of variables introduced in \cite{V} and then carefully studied in \cite{GMV} turns \eqref{elliptic} exactly into a related weighted Euclidean problem (at least in the radial setting), see Section \ref{weighted} for further details.

\noindent $\bullet$  Thanks to the above-mentioned existence result for a minimal, positive solution to \eqref{elliptic}, a corresponding strong asymptotic result  follows for solutions to \eqref{pme} (starting from a suitable class of initial data):
\[
\lim_{t\to + \infty} \left\| t^{\frac1{m-1}} u(\cdot,t) - v \right\|_{L^\infty(M)} = 0 \, .
\]
We refer the reader to Theorem \ref{thm:conv} below for the precise statement. The latter has the look of the theorems for solutions of the Dirichlet problem in a bounded Euclidean domain, but of course this gives significant information mainly inside a given ball.

\noindent $\bullet$ We also need  to better understand the behaviour in \emph{exterior} regions, where  we know that $u(x,t)=o(t^{-1/(m-1)})$ but we also know that there must be a finite bound for the support that needs being determined. Therefore,  we  investigate pointwise space-time bounds for solutions to \eqref{pme} in the spirit of \cite{GMV}. In fact, we shall prove (see Theorems \ref{thm:main} and \ref{thm:main:low } below) sharp upper and lower bounds on such solutions in terms of barriers which are similar, from an algebraic point of view, to the ones found in \cite{GMV} in the case $\mu<1$, given by formula \eqref{barrier1}. More specifically, for $\mu >1$ the long-time asymptotics of solutions is still approximately described by the profile \eqref{barrier1}, though inside the positive part signs are reversed.
The barrier is actually good for $r$ away from zero, and it immediately allows us to locate the free boundary, cf.~\eqref{bound-smooth-infty}. Thus, a remarkable formal similarity holds across the dividing value $\mu=1$, with quite different qualitative consequences. It does not hold of course at the very critical value $\mu=1$ as shown in \cite{GMV}.

\noindent $\bullet$ Another major contribution of this work concerns \it uniqueness \rm of the minimal solutions to problems \eqref{elliptic} and \eqref{pme}. This will be obtained in the class of solutions that, in an appropriate sense, tend to zero at (spatial) infinity, and will follow from comparison principles ``at infinity'', that have an independent interest. Uniqueness need not hold outside this class. See Theorem \ref{thm-ell-ex}, \it ii\rm) for the elliptic result, and Theorem \ref{uniqueness-parabolic} for the parabolic one. The proofs in the parabolic case are delicate applications of the \it duality methods \rm of  \cite{BCP,Pierre} and make essential use of the Dirichlet Green's function on geodesic balls, whereas the elliptic result follows from the parabolic ones. Some of the ideas are inspired from \cite{BK}. We notice that non-uniqueness results for solutions to the heat equation under the present curvature conditions are well-known: see e.g. \cite{IM, Grig, M1, I, I2, M2}.

We state all of the results here discussed in Section \ref{notation} and prove them in Section \ref{sect:proofs}.

\medskip

\noindent {\bf New results for weighted Euclidean Porous Medium Flows.} In \cite{V} a strong connection was found between Porous Medium flows on the hyperbolic space $\mathbb{H}^n$ and a class of weighted Porous Medium flows on the Euclidean space $\ren$. Actually, there exists a change of variable that transforms radially symmetric solutions of the first problem into solutions of the second, for a precise choice of the weight. This is interesting since the theory of weighted PME's is well developed in an independent way. The equivalence was then extended in \cite{GMV} to cover PME flows on a wide class of Cartan-Hadamard manifolds, i.e.~those mentioned above, and a corresponding large class of weighted PME's. In the final Section \ref{weighted} of this paper we exploit such change-of-variable technique  to obtain a result concerning the \emph{weighted Euclidean} PME
 \begin{equation}\label{pmeweighted}
\begin{cases}
\rho \, u_t=\Delta u^m  & \text{in } \mathbb{R}^n \times \mathbb{R}^+ \, , \\
u(\cdot,0)=u_0  & \text{in } \mathbb{R}^n  \, ,
\end{cases}
\end{equation}
for a class of (supercritical) weights $\rho$ which had not been addressed in \cite{GMV}. More precisely, the weights behave at infinity  like
\begin{equation*}
\rho(x) \sim \frac{C}{|x|^2 \, (\log |x|)^{\nu}}
\end{equation*}
for some $ \nu>1 $, where $|x|$ stands for Euclidean distance to the origin. In Section \ref{weighted} we show precise upper and lower bounds for solutions to \eqref{pmeweighted}, and prove a convergence result in terms of the separable solution constructed from the solution to the weighted nonlinear (sublinear in fact) elliptic problem
\begin{equation*}\label{elliptic.2}
-\Delta V^m = \frac{\rho(x)}{m-1} \, V \quad \text{in } \mathbb{R}^n \, .
\end{equation*}
This elliptic  equation is to be compared with \eqref{elliptic}.

\section{Preliminaries and statements of the results on manifolds}\label{notation}
As mentioned above, we shall work with general Cartan-Hadamard manifolds, but in order to prove our results it will be crucial to deal with the particular type called a Riemannian \emph{model manifold}. The latter is defined as a Riemannian manifold which is spherically symmetric with respect to a given pole $ o \in M $, and whose metric $ \mathrm{d}l $ is given by
\begin{equation*}\label{model}
\mathrm{d}l^2=\mathrm{d}r^2 + \psi^2(r) \, \mathrm{d}\theta^2
\end{equation*}
for a given (sufficiently) smooth, nonnegative function $\psi$ defined on $[0,+\infty)$ and satisfying $\psi(0)=0$, $ \psi'(0)=1$ (what we shall call a \emph{model function} from now on). Here,  $r=r(x)$ represents the geodesic distance $\operatorname{d}(x,o)$ between $o$ and a given point $x$, while $ \mathrm{d}\theta $ represents the standard metric on the sphere $ \mathbb{S}^{n-1} $. Because a model manifold is fully identified by $ \psi $, in such case we shall write $ M \equiv M_\psi $. In order to give a flavour of the class of manifolds we treat, consider the following explicit choice, given $\alpha>2$ and $ a,A>0 $:
\begin{equation*}\label{eq:psi}
\psi(r)=\begin{cases}
r &\text{if}\ r\in[0,\overline{r}] \, , \\
A\left(e^{a r^\alpha}-e^{a \overline{r}^\alpha}\right)+\overline{r} & \text{if}\ r>\overline{r} \, ,
\end{cases}
\end{equation*}
where $\overline{r}$ is the unique solution of $a A \alpha \, \overline{r}^{\alpha-1} e^{a \overline{r}^\alpha}=1 $, so that $ \psi $ is at least globally $ C^1 $. Upon denoting for short by $K(x)$ the sectional curvature at $x \in M_\psi $ with respect to any plane containing the radial direction, it is well known that on any model manifold $K(x)=-\psi''(r)/\psi(r)$, so that a trivial computation in our case yields
\[
K(x) \sim -a^2 \alpha^2\,r^{2\alpha-2} \ \ \ \textrm{as}\ r\to+\infty \, , \quad \text{i.e. } \mu=\alpha-1\,,
\]
and $\alpha>2$ is equivalent to $\mu>1$.

\noindent $\bullet$ One can use model manifolds $ M_\psi $ for (Laplacian) comparison with our original manifold $M$ as follows. Let us write the Laplace-Beltrami operator on $ M $ in polar coordinates (this is always possible, see e.g.~\cite[Section 2.2]{GMV}), namely
\begin{equation}\label{radial}
\Delta = \frac{\partial^2}{\partial r^2} + \mathsf{m}(r, \theta) \, \frac{\partial}{\partial r}+\Delta_{S_{r}}
 \,,
\end{equation}
where again $r:=\mathrm{d}(x,o)$, $o$ is a given (fixed) pole on $M$, $\theta\in{\mathbb S}^{n-1}$, $ S_r := \partial B_r(o) $, $ B_r(o) $ is the geodesic ball of radius $r$ on $M$ centered at $o$, $ \Delta_{S_{r}} $ is the Laplace-Beltrami operator on $ S_r $ and $ \mathsf{m} $ is an  appropriate function which coincides with  $ \Delta r$ away from the pole. In the case of a model manifold \eqref{radial} reads
\begin{equation}\label{radial-model}
\Delta = \frac{\partial^2}{\partial r^2} + (n-1) \, \frac{\psi^\prime(r)}{\psi(r)} \, \frac{\partial}{\partial r} + \frac{1}{\psi(r)^2} \, \Delta_{\mathbb{S}^{n-1}} \, .
\end{equation}
By classical results (see for instance the monograph \cite{GW}), if the \emph{sectional curvature} $ \mathrm{K}_\omega(x)$, with respect to any $2$-dimensional tangent subspace $ \omega $ at $x$ that contains the radial direction, satisfy the upper bound
\begin{equation*}
\mathrm{K}_\omega(x) \leq -\frac{\psi''(r)}{\psi(r)} \quad \text{for all } x\equiv (r,\theta)\in M\setminus\{o\}
\end{equation*}
for a given model function $ \psi $, then the Laplacian of the distance function on $M$ satisfies the lower bound
\begin{equation}\label{comp1}
\mathsf{m}(r,\theta) \geq (n-1) \, \frac{\psi'(r)}{\psi(r)}\quad \textrm{for all }  r>0 \, , \ \theta \in \mathbb{S}^{n-1} \, .
\end{equation}

\noindent $\bullet$ On the other hand, if we denote by $\mathrm{Ric}_{o}(x)$ the \emph{Ricci curvature} in the radial direction at $x$ and the lower bound
\[
\textrm{Ric}_{o}(x)\geq -(n-1)\,\frac{\psi''(r)}{\psi(r)} \quad \text{for all } x \equiv (r,\theta) \in M\setminus\{o\}
\]
holds for another given model function $\psi$, then the Laplacian of the distance function on $M$ satisfies the upper bound
\begin{equation*}\label{comp2}
\mathsf{m}(r, \theta) \leq (n-1) \, \frac{\psi'(r)}{ \psi(r)} \quad \text{for all } r>0 \, , \ \theta \in \mathbb S^{n-1}\,.
\end{equation*}
These two comparison principles are crucial to our goals since, as put in evidence by \eqref{radial}, $\mathsf{m}(r,\theta)$ is involved in the polar formula for the Laplacian of \emph{any} function, i.e.~not necessarily radial.

In the following, since integrals on $M$ will also appear, we shall denote by $ d \mu $ the Riemannian volume measure on $M$.


\subsection{Precise statements of the main results}\label{mainres}

We can now state in detail our main results: in order to simplify notations, at certain points we will write $ \liminf $, $ \limsup $ or $ \lim $ even though \emph{essential} limits would be more appropriate, as long as this does not create ambiguity. Since we shall often deal with limits as $ r \to +\infty $, we point out that they are always understood to be \emph{uniform} w.r.t.~$ \theta \in \mathbb{S}^{n-1} $.

The first result we address is concerned with the sublinear elliptic problem \eqref{elliptic}. Before, we need to provide appropriate definitions.

\begin{den}\label{defsolellipt}
We say that $v\in L^\infty(M)$, $v\ge0$, is a (very weak, or distributional) solution of the equation
\begin{equation}\label{eq:elliptic}
-\Delta {v}^m  = \frac 1{m-1} \, {v} \quad \text{in } M
\end{equation}
if it satisfies
\begin{equation}\label{veryweak}
- \int_M v^m \, \Delta \phi\, d\mu = \frac1{m-1}\int_M v \, \phi\, d\mu
\end{equation}
for all $\phi\in C^\infty_c(M)$. Similarly, we say that $v\in L^\infty(M)$, $v\ge0$, is a (very weak, or distributional) subsolution (supersolution) of equation \eqref{eq:elliptic} if \eqref{veryweak} holds with ``$=$'' replaced by ``$\leq$'' (``$\geq$''), for all $\phi\in C^\infty_c(M)$ with $ \phi\geq 0$.
\end{den}

\begin{thm}\label{thm-ell-ex}
Let $ M $ be an $n$-dimensional Cartan-Hadamard manifold of dimension $n\ge2$. Let a pole $ o \in M $ be fixed and $r=r(x):=\operatorname{d}(x,o)$.

\noindent(i) Assume that
\begin{equation}\label{hp-sect}
\mathrm{K}_\omega(x) \le -Q_1 \, r^{2 \mu_1} \quad \forall x \in M \setminus B_R(o)
\end{equation}
holds for some $\mu_1>1$ and $Q_1, R>0$. Then there exists a \emph{minimal} positive solution to the nonlinear elliptic equation \eqref{eq:elliptic},
which is regular, globally bounded and satisfies the upper estimate
\begin{equation}\label{eq:elliptic:above}
\limsup_{r\to+\infty} \, r^{\frac{\mu_1-1}{m-1}} \, v(x) \le  \frac{1}{\left[ m (\mu_1-1)(n-1) \sqrt{Q_1} \right]^{\frac{1}{m-1}}}  \, .
\end{equation}

\noindent(ii) The following \emph{comparison principle} holds: if $ \overline{v} $ is a positive, bounded supersolution of \eqref{eq:elliptic} and $ \underline{v} $ is a nonnegative, bounded subsolution of \eqref{eq:elliptic} such that
\begin{equation}\label{cond-elliptic-ss}
\liminf_{r \to +\infty} \, \left[ \overline{v}(x) - \underline{v}(x) \right] \ge 0 \, ,
\end{equation}
then $ \overline{v} \ge \underline{v} $ a.e.~in $ M $. In particular, under the validity of \eqref{hp-sect} the minimal solution $v$ is \emph{unique} in the class of nonnegative, nontrivial, bounded solutions satisfying
\begin{equation}\label{cond-elliptic}
\lim_{r \to + \infty} v(x)=0 \, .
\end{equation}
On the other hand, in this case uniqueness in the class of merely bounded solutions fails since for any $\alpha>0$ there exists a (unique) positive, bounded solution $v_\alpha$ of \eqref{eq:elliptic} such that
\begin{equation}\label{cond-elliptic-alpha}
\lim_{r \to +\infty} v_\alpha(x)=\alpha \, .
\end{equation}

\noindent (iii) If in addition to \eqref{hp-sect} there holds
\begin{equation}\label{hp-ricc}
\mathrm{Ric}_o(x) \ge -(n-1)\,Q_2 \,  r^{2 \mu_2} \quad \forall x \in M \setminus B_R(o)
\end{equation}
for some $\mu_2 \ge \mu_1 $ and $ Q_2,R >0 $, then $v$ also satisfies the \emph{matching} lower estimate
\begin{equation}\label{eq:elliptic:below}
\liminf_{r \to +\infty} \, r^{\frac{\mu_2-1}{m-1}} \, v(x) \ge  \frac{1}{\left[ m (\mu_2-1)(n-1) \sqrt{Q_2} \right]^{\frac{1}{m-1}}}  \, .
\end{equation}
As a consequence, the constants in the right-hand sides of \eqref{eq:elliptic:above} and \eqref{eq:elliptic:below} are \emph{sharp} on model manifolds $ M_\psi $ satisfying the curvature conditions \eqref{hp-sect} or \eqref{hp-ricc} as equalities outside a suitable ball.
\end{thm}

The solution $v$ of the elliptic problem \eqref{elliptic} has a crucial role in the study of the asymptotic behaviour of nonnegative solutions of the Porous Medium Equation \eqref{pme}. To this aim, let us start by giving the precise definition of solution we shall work with; since some of our results will also hold for sign-changing solutions, in such case the power $ u^m $ appearing in \eqref{pme}, as usual, is to be understood as $ |u|^{m-1}u $.

\begin{den}\label{defsol}
Let $ u_0 \in L^\infty(M) $. We say that $u\in L^\infty(M\times (0, +\infty))$ is a (very weak, or distributional) solution of problem \eqref{pme} if it satisfies
\begin{equation}\label{q50}
-\int_0^{+\infty} \int_M u\, \varphi_t\,  d\mu dt  = \int_0^{+\infty} \int_M u^m \, \Delta \varphi\, d\mu dt + \int_M u_0(x) \, \varphi(x,0)\, d\mu
\end{equation}
for all $\varphi\in C^\infty_c(M\times [0, +\infty))$. Similarly, we say that $u\in L^\infty(M\times (0, +\infty))$ is a (very weak, or distributional) subsolution (supersolution) of problem \eqref{pme} if \eqref{q50} holds with ``$=$'' replaced by ``$\leq$'' (``$\geq$''), for all $\varphi\in C^\infty_c(M\times [0, +\infty))$ with $ \varphi \ge 0 $.
\end{den}

The existence of a minimal solution to \eqref{pme}, for nonnegative bounded initial data, is a rather standard fact. Indeed, it can be obtained as a monotone limit, as $ R \to \infty $, of solutions $ u_R $ to the following auxiliary homogeneous Dirichlet problems on balls:
\begin{equation}\label{pme-balls}
\begin{cases}
\partial_t u_R = \Delta u_R^m  & \text{in } B_R(o) \times \mathbb{R}^+ \, , \\
u_R=0    & \text{on } \partial B_R(o) \times \mathbb{R}^+ \, , \\
u_R(\cdot,0)=u_0   & \text{in } B_R(o) \, .
\end{cases}
\end{equation}
By construction and comparison on balls, the map $ R \mapsto u_R $ is nondecreasing and therefore the limit object is well defined to be a bounded solution to \eqref{pme}, which is the smallest one among all nonnegative, bounded distributional solutions. For bounded and possibly sign-changing data, one can proceed e.g.~as pointed out in \cite[Section 3.3]{GMP} by using appropriate (signed) constants as barriers.

\begin{thm}\label{thm:conv}
Let $ M $ be a Cartan-Hadamard manifold of dimension $n\ge2$ satisfying the upper curvature bound \eqref{hp-sect} for some $\mu_1>1$ and $Q_1, R>0$. Let $u$ be the minimal solution to \eqref{pme} corresponding to any nonnegative, nontrivial initial datum $ u_0 \in L^\infty(M)$.  Let $v$ be the minimal, positive solution to \eqref{eq:elliptic} as in Theorem \ref{thm-ell-ex} (i). Then there holds
\begin{equation}\label{convergence}
\lim_{t \to + \infty} \left\| t^{\frac1{m-1}}u(\cdot,t)-v \right\|_{L^\infty(M)} = 0 \, .
\end{equation}
Moreover, $u$ satisfies the following universal upper bound:
\begin{equation}\label{universal}
u(x,t)\le \frac{v(x)}{t^{\frac1{m-1}}} \quad \text{for a.e. }(x,t) \in M \times \mathbb{R}^+ \, ;
\end{equation}
in particular, the \emph{absolute bound}
\begin{equation}\label{universal1}
\|u(t)\|_{L^\infty(M)}  \le \frac C{t^{\frac1{m-1}}} \quad \text{for a.e. } t>0
\end{equation}
holds for a suitable constant $C>0$ independent of $ u_0 $.
\end{thm}

In the present framework, the above-constructed minimal solution turns out to be the \emph{unique} solution in a rather large class of solutions, as the next result shows.

\begin{thm}\label{uniqueness-parabolic}
Let $ M $ be a Cartan-Hadamard manifold of dimension $n\ge2$. Let a pole $ o \in M $ be fixed and $r=r(x):=\operatorname{d}(x,o)$. Then the minimal solution to \eqref{pme} corresponding to a nonnegative, bounded initial datum $u_0$ is unique in the class of bounded solutions $u$ satisfying, for all $T,\epsilon$ such that $T>\epsilon>0$, the condition
\begin{equation}\label{cond-inf}
\lim_{r(x)\to+\infty}\int_\epsilon^T u(x,t)^{m}\,dt=0 \, .
\end{equation}
More generally, the comparison principle holds in the following form. Let $ \overline{u} $ be a bounded supersolution of \eqref{pme} and $ \underline{u} $ be a bounded subsolution of \eqref{pme} corresponding to the bounded initial data $\overline{u}_0$ and $\underline{u}_0$, respectively, with $ \overline{u}_0 \ge \underline{u}_0 $. Assume moreover that $ \overline{u} $ and $ \underline{u} $ satisfy
\begin{equation}\label{cond-inf-bis}
\liminf_{r(x) \to +\infty}  \int_\epsilon^T  \left[ \overline{u}(x,t)^m - \underline{u}(x,t)^m \right] dt \ge 0
\end{equation}
for all $T,\epsilon$ such that $T>\epsilon>0$. Then $ \overline{u} \ge \underline{u} $ a.e.~in $ M \times \mathbb{R}^+ $.
\end{thm}

\begin{oss}\rm
Condition \eqref{cond-inf} is crucial for uniqueness. Indeed, it is already known that uniqueness does not hold e.g.~in the class of merely bounded solutions, see \cite[Remark 3.12 and Theorem 3.8]{Pu} and \cite[Remark 2.4]{GMP}. Note that this fact also follows at once from our results, by considering the minimal solution to \eqref{pme} corresponding to $u_0\equiv1$, which cannot coincide with the constant solution in view of the absolute bound \eqref{universal1}. We also remark that our comparison principle bears, as a straightforward corollary, uniqueness results with additional conditions at infinity in the spirit e.g.~of \cite{GMPc,KP}. As concerns analogous problems in the context of weighted Euclidean spaces, we refer the reader to \cite{EK} for similar techniques in exterior domains of Euclidean space with a weight or to \cite{KT, KP} and references quoted therein for other types of weighted Euclidean problems.
\end{oss}

\begin{oss}\rm
Our main interest here is in nonnegative solutions, so we prefer to write our statements focusing on such solutions. Nevertheless, we observe that, by using both $ v/t^{1/(m-1)} $ and $ -v/t^{1/(m-1)}$ as barriers, it is not difficult to show that the universal bounds \eqref{universal} and \eqref{universal1} in fact also hold for sign-changing solutions, up to replacing $ u $ with $|u|$. Therefore, the same uniqueness result as in Theorem \ref{uniqueness-parabolic} also holds for bounded and sign-changing solutions.
\end{oss}

We can also prove explicit space-time upper and lower estimates for solutions to \eqref{pme} corresponding to compactly supported initial data, which in particular give information on the speed of propagation of the support.

\begin{thm}[Upper Estimates]\label{thm:main}
Let $ M $ be a Cartan-Hadamard manifold of dimension $n\ge2$. Let a pole $ o \in M $ be fixed and $r=r(x):=\operatorname{d}(x,o)$. Suppose that the upper curvature bound \eqref{hp-sect} holds for some $\mu_1>1$ and $Q_1, R>0$. Let $u$ be the minimal solution to \eqref{pme} corresponding to a nonnegative, bounded and compactly supported initial datum $ u_0$. Then the following pointwise estimate holds:
\begin{equation}\label{upper1}
u(x,t) \le \frac{C_1}{(t+t_1)^{\frac1{m-1}}}\left[\frac1{(r+r_1)^{\mu_1-1}} -
\frac{\gamma_1}{[\log(t+t_1)]^{\frac{\mu_1-1}{\mu_1+1}}}\right]^{\frac1{m-1}}_+
\end{equation}
for a.e.~$x \in M $ and $t \ge 0$, where $ C_1 , \gamma_1 , r_1 , t_1 $ are positive constants depending on $ n,m,\mu_1,Q_1,R,u_0 $. In particular, the estimate
\begin{equation}\label{bound-smooth-infty}
\overline{\mathsf{R}}(t) \le \kappa_1 \, (\log t)^{\frac{1}{1+\mu_1}}
\end{equation}
holds for a suitable $ \kappa_1 > 0 $ and all $t$ large enough, where $ \overline{\mathsf{R}}(t) $ stands for the radius of the smallest ball centered at $o$ that contains the support of the solution at time $t$.
\end{thm}

We comment that well-known smoothing effects (see \cite{GM3}) for the  here considered evolution show that \eqref{upper1} also holds for integrable and compactly supported initial data provided $t$ is large enough, since the solution becomes instantaneously bounded and stays compactly supported. Note also that the bound in \eqref{bound-smooth-infty} is formally the same as the one given in \eqref{radius} for the range $\mu\in(-1,1]$.

\begin{thm}[Lower Estimates]\label{thm:main:low }
Let $ M $ be a Cartan-Hadamard manifold of dimension $n\ge2$. Let a pole $ o \in M $ be fixed and $r=r(x):=\operatorname{d}(x,o)$. Suppose that the lower curvature bound \eqref{hp-ricc} holds for some $\mu_2>1$ and $Q_2, R>0$. Let $u$ be the minimal solution to \eqref{pme} corresponding to a nonnegative, nontrivial and compactly supported initial datum $ u_0 \in L^\infty(M) $. Then the following pointwise estimate holds:
\begin{equation}\label{lower1}
u(x,t) \ge \frac{C_2}{(t+t_2)^{\frac1{m-1}}}\left[\frac1{(r+r_2)^{\mu_2-1}}-
\frac{\gamma_2}{[\log(t+t_2)]^{\frac{\mu_2-1}{\mu_2+1}}}\right]^{\frac1{m-1}}_+
\end{equation}
for a.e.~$x \in M \setminus B_{1}(o)$ and $t \ge 0 $, where $ C_2 , \gamma_2 , r_2 , t_2 $ are positive constants depending on $ n,m,\mu_2,Q_2,R,u_0 $ and (locally) on the metric of $M$. In particular, the estimates
\begin{equation}\label{bound-smooth-infty-lower}
\| u(t) \|_\infty \ge \frac{K}{t^{\frac1{m-1}}} \quad \text{and} \quad \underline{\mathsf{R}}(t) \ge \kappa_2 \, (\log t)^{\frac{1}{1+\mu_2}}
\end{equation}
hold for suitable $ K , \kappa_2 >0 $ and all $t$ large enough, where $ \underline{\mathsf{R}}(t) $ stands for the radius of the largest ball that is contained in the support of the solution at time $t$.
\end{thm}

We observe again that the above-mentioned smoothing effects for the equation at hand ensure that in fact \eqref{lower1} also holds e.g.~for nonnegative, nontrivial integrable initial data, provided $t$ is large enough.

\begin{oss}\rm
It is apparent that the upper and lower estimates stated above match if the corresponding upper and lower curvature bounds match, i.e.~if $\mu_1=\mu_2=\mu>1$. In such case, the resulting two-sided estimate falls within the so-called \emph{global Harnack principles}, following the terminology of \cite{DKV, BV, GMV}. Analogous global Harnack principles had been established in \cite{GMV} in the cases $\mu\in(-\infty,1]$.
\end{oss}

\section{Proofs of the results on manifolds}\label{sect:proofs}

We start by stating two preliminary lemmas related to the Laplacian-comparison discussion of Section \ref{notation}, which will be very useful and whose proofs are completely analogous to the ones of Lemmas 4.1 and 4.2 in \cite{GMV}, respectively.
\begin{lem}\label{lem: ode-comparison}
Let $R,Q>0$ and $ \mu >1$ be fixed parameters. Let $ \psi $ be the solution to the following ODE:
\begin{equation}\label{ode-psi-upper-lemma}
\psi''(r) = w(r) \, \psi(r) \quad \forall r>0 \, , \quad \psi(0)=0  \, , \quad \psi'(0)=1   \, ,
\end{equation}
where
\begin{equation*}\label{ode-psi-upper-f-lemma}
w(r):=
\begin{cases}
0 & \forall r \in [0,R] \, , \\
\frac{Q \, (2R)^{2\mu}}{R} \, (r-R) & \forall r \in (R,2R] \, , \\
Q \, r^{2\mu} & \forall r > 2R \, .
\end{cases}
\end{equation*}
Then $ \psi \in C^2([0,+\infty)) $, $ \psi $ is positive on $ (0,+\infty) $ with $ \lim_{r \to +\infty} \psi(r)=+\infty $, $ \psi^\prime $ is positive on $ [0,+\infty) $ and there holds
\begin{equation}\label{ode-psi-estimate-fond-lemma}
\lim_{r \to +\infty} \frac{\psi^\prime(r)}{r^\mu \, \psi(r)} = \sqrt{Q} \, .
\end{equation}
\end{lem}

\begin{lem}\label{lem: ode-comparison-2}
Let $Q,D>0$ and $ \mu >1 $ be fixed parameters. Let $ \psi $ be the solution to the same ODE as in \eqref{ode-psi-upper-lemma}, with
\begin{equation*}\label{ode-psi-lower-f-lemma-bis}
w(r):= \max \left\{ D , Q \, r^{2\mu} \right\} .
\end{equation*}
Then $ \psi \in C^2([0,+\infty)) $, $ \psi $ is positive on $ (0,+\infty) $ with $ \lim_{r \to +\infty} \psi(r)=+\infty $, $ \psi^\prime $ is positive on $ [0,+\infty) $ and \eqref{ode-psi-estimate-fond-lemma} holds.
\end{lem}

\subsection{The nonlinear elliptic equation: existence and asymptotic properties}

In order to prove existence and asymptotic properties of the minimal positive solution to equation \eqref{eq:elliptic}, we shall proceed by means of careful barrier arguments. Throughout, we shall tacitly adopt the same notations as in Section \ref{notation}. On the other hand, the uniqueness statements will be proved in Section \ref{uniq-p}, by exploiting the parabolic results.

\begin{lem}\label{ell-ssol}
Let $ M $ be a Cartan-Hadamard manifold of dimension $n\ge2$, satisfying the upper curvature bound \eqref{hp-sect} for some $\mu_1>1$ and $Q_1, R>0$. Then there exists $ r_0=r_0(n,m,\mu_1,Q_1,R)>0$ such that the function
\begin{equation*}\label{ell-ssol-1}
\overline{v}(x) \equiv \overline{v}(r):=\frac{C}{\left( r^2 + r_0^2 \right)^{\frac{\mu_1-1}{2(m-1)}}}
\end{equation*}
is a supersolution of \eqref{eq:elliptic} for all $C=C(n,m,\mu_1,Q_1,R) $ sufficiently large.
\end{lem}
\begin{proof}
In view of the curvature bound, Lemma \ref{lem: ode-comparison} (with $ \mu=\mu_1 $ and $ Q=Q_1 $) and the Laplacian-comparison results recalled in Section \ref{notation} (concerning $ \mathrm{K}_\omega$), we can assert that there exists a constant $ \beta>0 $ (depending on $ n, \mu_1, Q_1, R $) such that
\begin{equation}\label{eq:lap-comp}
\mathsf{m}(r,\theta) \ge \frac{n-1}{r} \quad \text{in } M \, , \quad \mathsf{m}(r,\theta) \ge \beta \, r^{\mu_1}  \quad \text{in } M \setminus B_1(o) \, .
\end{equation}
For the sake of better readability, from now on we set $ Q_1 \equiv Q $ and $ \mu_1 \equiv \mu $. Because $ \overline{v} $ is radially decreasing, we can replace $ \mathsf{m}(r,\theta) $ with the lower bounds \eqref{eq:lap-comp} in the supersolution inequality. The radial derivatives of $ \overline{v}^m $ read
\begin{equation*}\label{eq:rad-der}
\begin{aligned}
\left( \overline{v}^m \right)_r = & - C^m \, \frac{m(\mu-1)}{m-1} \, \frac{r}{\left( r^2 + r_0^2 \right)^{\frac{m(\mu-1)}{2(m-1)}+1}} \, , \\
\left( \overline{v}^m \right)_{rr} = & -C^m \, \frac{m(\mu-1)}{m-1} \, \frac{1}{\left( r^2 + r_0^2 \right)^{\frac{m(\mu-1)}{2(m-1)}+1}} \left[ 1 - \left( \frac{m(\mu+1)-2}{m-1} \right) \frac{r^2}{r^2 + r_0^2} \right] .
\end{aligned}
\end{equation*}
Therefore, in order to make $ \overline{v} $ a supersolution in $ B_1(o) $, it is enough to ask
\begin{equation}\label{eq:rad-der-1}
C^{m-1} \left[ n - \frac{[m(\mu+1)-2] \, r^2}{(m-1)(r^2+r_0^2)} \right] \ge \frac{\left( r^2 + r_0^2 \right)^{\frac{\mu+1}{2}}}{m(\mu-1)} \quad \forall r \in (0,1) \, .
\end{equation}
It is immediate to check that inequality \eqref{eq:rad-der-1} is equivalent to
\begin{equation}\label{eq:rad-der-2}
C^{m-1} \left[ n - \frac{m(\mu+1)-2}{(m-1)(1+r_0^2)} \right] \ge \frac{\left( 1 + r_0^2 \right)^{\frac{\mu+1}{2}}}{m(\mu-1)} \, .
\end{equation}
Similarly, in order to make $ \overline{v} $ a supersolution in $ M \setminus B_1(o) $, it suffices to require
\begin{equation*}\label{eq:rad-der-3}
C^{m-1} \left[ \beta - \frac{m(\mu+1)-2}{(m-1)\,r^{\mu-1}\,(r^2+r_0^2)} \right] \ge \frac{\left( r^2 + r_0^2 \right)^{\frac{\mu+1}{2}}}{m(\mu-1)\,r^{\mu+1}} \quad \forall r \ge 1 \, ,
\end{equation*}
which is again equivalent to the validity of the inequality itself at $r=1$, that is
\begin{equation}\label{eq:rad-der-4}
C^{m-1} \left[ \beta - \frac{m(\mu+1)-2}{(m-1)(1+r_0^2)} \right] \ge \frac{\left( 1 + r_0^2 \right)^{\frac{\mu+1}{2}}}{m(\mu-1)} \, . 
\end{equation}
It is then apparent that one can choose $ r_0=r_0(n,m,\mu,\beta) $ so large that for all $ C=C(n,m,\mu,\beta,r_0) $ large enough both conditions \eqref{eq:rad-der-2} and \eqref{eq:rad-der-4} are fulfilled.
\end{proof}

\begin{lem}\label{lemma-lower}
Let $ M $ be a Cartan-Hadamard manifold of dimension $n\ge2$, satisfying the lower curvature bound \eqref{hp-ricc} for some $\mu_2>1$ and $Q_2, R>0$. Then the function
\begin{equation*}\label{sotto-sol}
\underline{v}(x) \equiv \underline{v}(r) := C_\varepsilon \left[ \frac{1}{\left( r^2 + r_0^2 \right)^{\frac{\mu_2-1}{2}}} - \delta \right]_+^{\frac{1}{m-1}} , \quad C_\varepsilon^{m-1} := \frac{1}{m (\mu_2-1) (n-1)\sqrt{Q_2+\varepsilon}} \, ,
\end{equation*}
is a subsolution of \eqref{eq:elliptic} for all $ \varepsilon,\delta>0$ and all $ r_0 $ large enough (depending on $ \varepsilon $ but independent of $ \delta $).
\end{lem}
\begin{proof}
As above, we set $ Q_2 \equiv Q $ and $ \mu_2 \equiv \mu $. The radial derivatives of $ \underline{v}^m $ then read (we let $ C_\varepsilon \equiv C$ free for the moment)
\begin{equation*}\label{eq:rad-der-ssol}
\begin{aligned}
\left( \underline{v}^m \right)_r = & - C^m \, \frac{m(\mu-1)}{m-1} \, \frac{r}{\left( r^2 + r_0^2 \right)^{\frac{\mu+1}{2}}}  \left[ \frac{1}{\left( r^2 + r_0^2 \right)^{\frac{\mu-1}{2}}} - \delta \right]_+^{\frac{1}{m-1}} , \\
\left( \underline{v}^m \right)_{rr} = & -C^m \, \frac{m(\mu-1)}{m-1} \left[ \frac{1}{\left( r^2 + r_0^2 \right)^{\frac{\mu+1}{2}}} -\frac{(\mu+1) \, r^2}{\left( r^2 + r_0^2 \right)^{\frac{\mu+3}{2}}} \right] \left[ \frac{1}{\left( r^2 + r_0^2 \right)^{\frac{\mu-1}{2}}} - \delta \right]_+^{\frac{1}{m-1}} \\
 & + C^m \, \frac{m(\mu-1)^2}{(m-1)^2} \, \frac{r^2}{\left( r^2 + r_0^2 \right)^{\mu+1}}  \left[ \frac{1}{\left( r^2 + r_0^2 \right)^{\frac{\mu-1}{2}}} - \delta \right]_+^{\frac{2-m}{m-1}} .
\end{aligned}
\end{equation*}
By virtue of the curvature bound \eqref{hp-ricc}, Lemma \ref{lem: ode-comparison-2} and the Laplacian-comparison results recalled in Section  \ref{notation} (concerning $ \mathrm{Ric}_o $), we can infer that for all $ \varepsilon>0 $ there exist two constants $ \alpha, \underline{R} >0 $, depending on $ n,\mu,Q,R,\varepsilon$ and on
\[D:=-\inf_{x \in B_R(o)}  \mathrm{Ric}_o(x)/(n-1) \, , \]
such that
\begin{equation*}\label{eq:lap-comp-below}
\mathsf{m}(r,\theta) \le \frac{\alpha}{r} \quad \text{in } B_{\underline{R}}(o) \, , \quad \mathsf{m}(r,\theta) \le (n-1)\sqrt{Q +\varepsilon} \, r^\mu  \quad \text{in } M \setminus B_{\underline{R}}(o) \, .
\end{equation*}
Hence, in order to make $ \underline{v} $ a subsolution in $ B_{\underline{R}}(o) $, it is enough to ask (note that $ \underline{v} $ is still radially decreasing)
\begin{equation*}\label{eq:ssol-1}
\begin{aligned}
& C^{m-1} \left[ \alpha+1 - \frac{(\mu+1)r^2}{r^2+r_0^2} \right] \left[ \frac{1}{\left( r^2 + r_0^2 \right)^{\frac{\mu-1}{2}}} - \delta \right]_+ - \frac{\mu-1}{m-1} \, \frac{r^2}{\left( r^2 + r_0^2 \right)^{\frac{\mu+1}{2}}} \\
\le & \frac{\left( r^2 + r_0^2 \right)^{\frac{\mu+1}{2}}}{m(\mu-1)} \left[ \frac{1}{\left( r^2 + r_0^2 \right)^{\frac{\mu-1}{2}}} - \delta \right]_+ \quad \forall r \in (0,\underline{R}) \, ,
\end{aligned}
\end{equation*}
which is implied e.g.~by
\begin{equation}\label{eq:ssol-2}
C^{m-1} \, (\alpha+1) \le \frac{r_0^{\mu+1}}{m(\mu-1)} \, .
\end{equation}
On the other hand, $ \underline{v} $ is a subsolution in $ M \setminus B_{\underline{R}}(o) $ provided
\begin{equation*}\label{eq:ssol-3}
\begin{aligned}
& C^{m-1} \left[ 1 - \frac{(\mu+1) \, r^2}{(r^2+r_0^2)\left[(n-1)\sqrt{Q+\varepsilon} \, r^{\mu+1}+1\right]} \right] \left[ \frac{1}{\left( r^2 + r_0^2 \right)^{\frac{\mu-1}{2}}} - \delta \right]_+ \\
& - \frac{(\mu-1) \, r^2}{(m-1)\left[ (n-1)\sqrt{Q+\varepsilon} \, r^{\mu+1}+1 \right] \left( r^2 + r_0^2 \right)^{\frac{\mu+1}{2}}} \\
\le & \frac{\left( r^2 + r_0^2 \right)^{\frac{\mu+1}{2}}}{m(\mu-1)\left[(n-1)\sqrt{Q+\varepsilon} \, r^{\mu+1}+1 \right]} \left[ \frac{1}{\left( r^2 + r_0^2 \right)^{\frac{\mu-1}{2}}} - \delta \right]_+ \quad \forall r \ge \underline{R} \, ,
\end{aligned}
\end{equation*}
which is straightforwardly satisfied e.g.~if
\begin{equation}\label{eq:ssol-4}
C^{m-1} \le \, \frac{1}{m (\mu-1)(n-1)\sqrt{Q+\varepsilon}} \, , \quad r_0^2 \ge \frac{1}{(n-1)\sqrt{Q+\varepsilon} \, \underline{R}^{\mu-1}} \, .
\end{equation}
Note that $ C^{m-1} $ can be picked exactly equal to the right-hand side~of the first inequality in \eqref{eq:ssol-4}, provided $ r_0 $ is so large that also \eqref{eq:ssol-2} holds with such choice.
\end{proof}
Having at our disposal the barriers constructed in Lemmas \ref{ell-ssol}--\ref{lemma-lower}, we are now in position to prove our main result as regards existence and asymptotic properties of solutions to equation \eqref{eq:elliptic}.
\begin{proof}[Proof of Theorem \ref{thm-ell-ex} (i), (iii)]
To begin with, for all $ k \in \mathbb{N} $ one starts by considering the positive solutions to the Dirichlet problems
\begin{equation}\label{elliptic-k}
\begin{cases}
-\Delta v_k^m = \frac{1}{m-1} \, v_k & \text{in } B_k(o) \, , \\
v_k=0  & \text{on } \partial B_k(o) \, .
\end{cases}
\end{equation}
Existence and inner regularity can be shown by standard techniques e.g.~as in \cite{BK}, whose methods are rather general and apply to the present situation as well. If the curvature bound \eqref{hp-sect} holds, thanks to the comparison principle for subsolutions and positive supersolutions of \eqref{elliptic-k}, for which we refer again to \cite{BK}, we can assert that $ v_k \le v_{k+1} $ in $ B_k $ and, by virtue of Lemma \ref{ell-ssol}, the validity of the estimate
\begin{equation*}\label{est:above}
v_k(x) \le \frac{C}{\left( r^2 + r_0^2 \right)^{\frac{\mu_1-1}{2(m-1)}}} \quad \forall x \in B_k(o) \, .
\end{equation*}
Therefore, we can pass to the limit monotonically as $ k \to \infty $ to obtain the  existence of a positive solution $v$ to \eqref{eq:elliptic}, which satisfies
\begin{equation*}\label{est:above-limit}
v(x) \le \frac{C}{\left( r^2 + r_0^2 \right)^{\frac{\mu_1-1}{2(m-1)}}} \quad \forall x \in M
\end{equation*}
with the same constants $ r_0,C >0 $ as in Lemma \ref{ell-ssol}. The fact that such solution is minimal in the class of positive (and locally bounded) solutions is a direct consequence of the construction procedure. Regularity then follows again by positivity and standard elliptic estimates.

We are left with proving the upper estimate \eqref{eq:elliptic:above}. To this end, consider the function $\psi$ given in Lemma \ref{lem: ode-comparison} with the choices $\mu=\mu_1$, $Q=Q_1$, and denote by $ M_\psi $ the Riemannian model associated with it. By construction, as in the proof of Lemma \ref{ell-ssol}, the upper curvature bound \eqref{hp-sect} ensures that $M_\psi$ can be used for Laplacian comparison, namely \eqref{comp1} holds on $ M $. It is apparent that on $M_\psi$ the solution $v$ of \eqref{eq:elliptic}, constructed exactly as above, can be assumed to be radial. Let us write down explicitly the ordinary differential equation solved by $v(r)\equiv v(x) $ (by using \eqref{radial-model}), which reads
\begin{equation}\label{model-elliptic}
\left(\psi^{n-1}\left(v^m\right)'\right)' = -\frac{\psi^{n-1}}{m-1} \, v \quad \text{in } (0,+\infty) \, ,
\end{equation}
where $ ^\prime $ denotes derivation with respect to $ r $. Because $v$ is positive, hence regular, we can integrate \eqref{model-elliptic} to get
\begin{equation}\label{integral--1}
\left(v^m\right)'\!(r)=-\frac1{(m-1)\,\psi^{n-1}(r)}\int_0^r\psi^{n-1}(s) \, v(s)\,{ d}s \quad \forall r>0 \, .
\end{equation}
Upon integrating further \eqref{integral--1} from $r$ to infinity (thanks to Lemma \ref{ell-ssol} we already know that $v(r)$ vanishes $r\to\infty$), we end up with
\begin{equation}\label{integral}
v^m(r)=\frac1{m-1}\int_r^{+\infty} \left(\frac1{\psi^{n-1}(t)}\int_0^t\psi^{n-1}(s) \, v(s)\,{ d}s\right) {d}t \quad \forall r>0 \, .
\end{equation}
Suppose now that the a priori estimate
\begin{equation}\label{limsup}
\limsup_{r \to +\infty} \, r^{\frac{\mu_1-1}{m-1}} \, v(r) \le C
\end{equation}
holds for some $C>0$. We know, still from Lemma \ref{ell-ssol}, that \eqref{limsup} does hold for a suitable $C>0$; moreover, Lemma \ref{lemma-lower} entails that
\begin{equation}\label{low}
C^{m-1}\ge\frac{1}{m (\mu_1-1) (n-1)\sqrt{Q_1}} \, ,
\end{equation}
since $M_\psi$ also satisfies the curvature bound from below \eqref{hp-ricc} with $ \mu_2=\mu_1 $ and $ Q_2=Q_1 $. Let us insert the asymptotic information given by \eqref{limsup} into identity \eqref{integral}. So, by virtue of \eqref{limsup}, l'H\^{o}pital's rule and \eqref{ode-psi-estimate-fond-lemma}, it is not difficult to infer that for every $ \varepsilon > 0 $ there exists $ R_\varepsilon> 0 $ such that
\begin{equation}\label{integral1}
 \int_0^t\psi^{n-1}(s) \, v(s)\,{ d}s \le \frac{C+\varepsilon}{(n-1)\sqrt{Q_1}} \, \psi^{n-1}(t) \, t^{-\frac{m\mu_1-1}{m-1}} \quad \forall t\ge R_\varepsilon \, .
\end{equation}
Thus, by plugging the bound \eqref{integral1} in \eqref{integral} we obtain the estimate
\[
v^m(r)\le \frac{C+\varepsilon}{m(\mu_1-1)(n-1)\sqrt{Q_1}} \, r^{-m \, \frac{\mu_1-1}{m-1}} \quad \forall r\ge R_\varepsilon \, .
\]
It follows that the constant $C$ in \eqref{limsup} is improved by
\[
\left[ \frac{C}{m(\mu_1-1)(n-1)\sqrt{Q_1}} \right]^{\frac1m} .
\]
Notice that the validity of the bound from below \eqref{low} ensures that this latter constant is actually not larger than $C$. By iterating such procedure, it is straightforward to show that \eqref{limsup} eventually holds on $ M_\psi $ with
$$
C = \frac{1}{\left[ m (\mu_1-1)(n-1) \sqrt{Q_1} \right]^{\frac{1}{m-1}}} \, .
$$
In fact we have proved more, namely the last part of the statement: estimate \eqref{limsup} is satisfied with  equality as a limit  on $ M_\psi $. The validity of \eqref{eq:elliptic:above} on the original manifold $ M $ is then a simple consequence of Laplacian comparison, since the just-constructed solution $v$ becomes a \emph{supersolution} on $M$ (here it is key that $ v $ is radially decreasing).

Suppose now that also the lower bound on the Ricci curvature \eqref{hp-ricc} holds. By arguing as above, we can deduce that the positivity of $ v_k $ in $ B_k $, comparison and Lemma \ref{lemma-lower} yield
\begin{equation}\label{est:below}
v_k(x) \ge \left[ \left( \frac{1}{m (\mu_2-1) \sqrt{Q_2+\varepsilon}} \right) \left[ \frac{1}{\left( r^2 + r_0^2 \right)^{\frac{\mu_2-1}{2}}} - \delta \right] \right]_+^{\frac{1}{m-1}} \quad \forall x \in B_k \, ,
\end{equation}
provided $ k $ is so large  that the support of the r.h.s.~is contained in $ B_k $. Hence, estimate \eqref{eq:elliptic:below} just follows by letting first $ k \to \infty $, then $ \delta \to 0 $, then $ r \to +\infty $ and finally $ \varepsilon \to 0 $ in \eqref{est:below}.
\end{proof}

\subsection{Comparison principles, uniqueness and non-uniqueness results}\label{uniq-p}
Our strategy of proof as concerns the comparison principle ``at infinity'' for \eqref{pme} (and as a consequence for \eqref{elliptic}) borrows some ideas and techniques both from \cite{BK} and \cite{GMP}. To some extent, we combine a celebrated ``duality method'', which goes back to \cite{BCP,Pierre}, and suitable estimates that take advantage of the Green function $G_R$ on the Riemannian balls $B_R(o)$ associated with the Dirichlet Laplace-Beltrami operator. In a few words, such a method consists in choosing a smart enough test function, which is in fact the solution of a proper dual problem, to be plugged in the (very) weak formulation satisfied by $ \overline{u}-\underline{u} $, $ \overline{u}$ and $\underline{u}$ being super- and subsolutions respectively. Then one exploits $G_R$ in order to bound the normal derivative of the dual solution (in particular its integral over $\partial B_R(o)$): this allows us to prove that remainder terms arising from the integrations by parts carried out vanish as $ R \to \infty $. The integral inequality we obtain in the end has, as a free parameter, the final datum of the dual problem; since the latter is arbitrary, the claimed inequality follows. We stress that we use only Green functions on balls, hence nonparabolicity of the manifold need not be assumed.

\begin{proof}[Proof of the parabolic uniqueness and comparison: Theorem \ref{uniqueness-parabolic}]
Since $ \overline u $ and $ \underline u$ are a distributional supersolution and subsolution, respectively, of \eqref{pme}, it is straightforward to check that they satisfy the following inequality:
\begin{equation}\label{n3}
\int_0^T \int_M \left[ (\overline{u} - \underline{u}) \, \xi_t + \left(\overline{u}^m - \underline{u}^m\right) \Delta \xi \right] d\mu dt \le \int_M \left[ \overline{u}(x,T) - \underline{u}(x,T) \right] \xi(x,T) \, d\mu \, ,
\end{equation}
for a.e.~$ T>0 $ and for any nonnegative $\xi\in C^\infty_c(M\times [0, T])$. Let
\begin{equation}\label{n4}
a(x,t):=
\begin{cases}
\frac{\overline{u}(x,t)^m - \underline{u}(x,t)^m}{\overline{u}(x,t) - \underline{u}(x,t)} & \text{if } \overline u (x,t) \neq \underline u(x,t) \, , \\
0 & \text{if } \overline u(x,t) = \underline u(x,t) \, .
\end{cases}
\end{equation}
Clearly $a$ is everywhere nonnegative and bounded. So, due to \eqref{n4}, inequality \eqref{n3} can be rewritten as
\begin{equation*}
\int_0^T \int_M (\overline u - \underline u) \left(\xi_t + a \, \Delta \xi \right) d\mu dt \le \int_M \left[ \overline u(x,T) - \underline u(x,T) \right] \xi(x,T) \, d\mu \, .
\end{equation*}
Now let us take $R_0>0$ and $R>R_0+1$: in the sequel, $R_0$ is meant to be fixed once for all, while $R$ will eventually go to infinity. In order to simplify notations, we set $ B_r(o) \equiv B_r $ for all $ r>0 $. Let
\begin{equation*}
\begin{gathered}
\{a_{n, R}\}_{n\in \mathbb N} \equiv  \{a_n\}_{n\in \mathbb N}\subset C^\infty(M\times [0, T])\,,\quad a_n>0 \quad \text{in}\ B_R \ \textrm{for every}\ n\in \mathbb N \, , \\
\omega\in C^\infty_c(M) \, , \quad \omega \geq  0 \, , \quad \omega \equiv 0 \quad  \text{in } M \setminus B_{R_0} \, .
\end{gathered}
\end{equation*}
For every $n\in \mathbb N$, we call $\xi_n$ the solution of the backward parabolic (dual) problem
\begin{equation}\label{n9}
\begin{cases}
\partial_t \xi_n  + a_n \, \Delta \xi_n = 0 & \text{in } B_R \times (0, T) \, , \\
\xi_n = 0 & \text{on } \partial B_R\times (0, T) \, , \\
\xi_n = \omega & \text{in } B_R \times \{T\} \, .
\end{cases}
\end{equation}
It is plain that $ \xi_n $ is positive in $ B_R \times (0,T) $, regular and satisfies
\begin{equation}\label{n43}
\frac{\partial \xi_n}{\partial \nu}  \leq 0 \quad \text{on } \partial B_R\times (0, T)\,.
\end{equation}
For all $0<\varepsilon< 1 / 2$ we can pick a family of nonincreasing cut-off functions $ \tilde \phi_\varepsilon $ such that
\begin{equation}\label{n10}
\begin{gathered}
\tilde \phi_\varepsilon\in C^\infty_c([0,\infty))\,,\quad 0\leq \tilde\phi_\varepsilon\leq 1 \, , \quad
\tilde \phi_\varepsilon \equiv 1 \quad \text{in}\ [0, R-2\varepsilon) \, , \quad \tilde \phi_\varepsilon \equiv 0 \quad \text{in}\ (R-\varepsilon, \infty) \, , \\
\end{gathered}
\end{equation}
and then set
\begin{equation}\label{n11}
\phi_\varepsilon(x):=\tilde \phi_\varepsilon(r(x))\quad \forall x\in M \, .
\end{equation}
We are therefore allowed to plug $ \xi \equiv \xi_n  \phi_\varepsilon $ in \eqref{n3}, which yields
\begin{equation}\label{n14}
\begin{aligned}
& \int_0^T \int_M (\overline{u} - \underline{u}) \, \phi_\varepsilon \, (a- a_n) \, \Delta \xi_n \, d\mu dt  \\
& + \int_0^T \int_M (\overline{u}^m - \underline{u}^m) \left(2\langle \nabla \phi_\varepsilon, \nabla \xi_n\rangle + \xi_n \, \Delta \phi_\varepsilon\right) d\mu dt \\
\le & \int_M [\overline{u}(x,T) - \underline{u}(x,T)] \, \omega(x) \, \phi_\varepsilon(x) \, d\mu \, .
\end{aligned}
\end{equation}
Let us set
\begin{equation}\label{n15}
I_{n , \varepsilon}(R):= \int_0^T \int_M (\overline{u}-\underline{u}) \, \phi_\varepsilon \, (a- a_n) \, \Delta\xi_n \, d\mu dt \, ,
\end{equation}
\begin{equation}\label{n16}
J_{n , \varepsilon}(R):= \int_0^T \int_M \left(\overline{u}^m - \underline{u}^m\right) \left(2\langle \nabla \phi_\varepsilon, \nabla \xi_n\rangle + \xi_n \, \Delta \phi_\varepsilon\right) d\mu dt \, .
\end{equation}
By standard computations, thanks to \eqref{n10}--\eqref{n11}, the positivity and regularity of $ \xi_n $ as well as the fact that the latter vanishes on $ \partial B_R $, it is not difficult to show that (recall also \eqref{n43})
\begin{equation}\label{lim-dn}
\begin{gathered}
\lim_{\varepsilon \downarrow 0} \int_{M} \max_{t \in [0,T]}  \left(2\langle \nabla \phi_\varepsilon, \nabla \xi_n\rangle + \xi_n \, \Delta \phi_\varepsilon\right)_+ d\mu =   \int_{S_R} \max_{t \in [0,T]} \left|\frac{\partial \xi_n}{\partial \nu}(x,t)\right| d \sigma \, ,\\
\lim_{\varepsilon \downarrow 0}  \int_{M} \max_{t \in [0,T]} \left(2\langle \nabla \phi_\varepsilon, \nabla \xi_n\rangle + \xi_n \, \Delta \phi_\varepsilon\right)_- d\mu = 0 \, ,
\end{gathered}
\end{equation}
where $ S_R := \partial B_R $ and $ d\sigma $ is the Hausdorff measure on $ \partial S_R $. By virtue of \eqref{cond-inf-bis}, for every $ \epsilon \in (0,T) $ we know that there exists some positive function $ \alpha_\epsilon(R) $ such that $ \lim_{R\to\infty} \alpha_\epsilon(R) = 0 $ and
\begin{equation}\label{eq:est-liminf}
\int_\epsilon^T \left[ \overline{u}(x,t)^m - \underline{u}(x,t)^m \right] dt \ge - \alpha_\epsilon(R) \quad \text{for a.e.} \ x \in M \setminus B_{R-1} \, ;
\end{equation}
on the other hand, $ \overline{u}$ and $\underline{u} $ being globally bounded,
\begin{equation}\label{eq:est-liminf-2}
\int_0^\epsilon \left(|\overline u(x,t)|^m+|\underline u(x,t)|^m\right) dt \le \left( \| \overline{u} \|_\infty + \| \underline{u} \|_\infty \right) \epsilon =: C \epsilon \quad \text{for a.e.} \ x \in M \setminus B_{R-1} \, .
\end{equation}
Thus, by combining \eqref{n16}, \eqref{lim-dn}, \eqref{eq:est-liminf} and \eqref{eq:est-liminf-2}, we end up with
\begin{equation}\label{eq:est-comp-1}
J_n(R) := \liminf_{\varepsilon \downarrow 0} J_{n,\varepsilon}(R) \ge - \left( \alpha_\epsilon(R) + C \epsilon \right) \int_{S_R} \max_{t \in [0,T]} \left|\frac{\partial \xi_n}{\partial \nu}(x,t)\right| d \sigma \, .
\end{equation}
We now need to provide a suitable bound on
$$
\left|\frac{\partial \xi_n}{\partial \nu}(x,t)\right| \quad \text{for every } x \in S_R \text{ and } t \in (0,T) \, .
$$
To this aim, let us denote by $ x \mapsto G_R(x,o) $ the Green function of the Dirichlet Laplace-Beltrami operator on $ B_R $, namely the unique positive solution to
\begin{equation}\label{eq:green}
\begin{cases}
-\Delta G_R(x,o) = \delta_{o} & \text{in } B_R \, , \\
G_R(x,o) = 0  & \text{on } S_R \, ,
\end{cases}
\end{equation}
$ \delta_o $ being the Dirac mass centered at $ o $. Note that, upon integrating the differential equation in \eqref{eq:green} and observing that $ \tfrac {\partial G_R} {\partial \nu} \le 0 $ on $ S_R $, we end up with
\begin{equation}\label{eq:green-1}
\int_{S_R} \left| \frac{\partial G_R }{\partial \nu}  \right| d\sigma = 1 \, ,
\end{equation}
a crucial identity we shall exploit below. It is well known that $ R \mapsto G_R(x,o) $ is nondecreasing and $ x \mapsto G_R(x,o) $ is strictly positive in $ B_R $: in particular there exists $ \lambda>0 $, independent of $ R $, such that
\begin{equation}\label{eq:green-2}
 \lambda \, G_R(x,o) \ge \| \omega \|_\infty \quad \forall x \in S_{R_0} \, , \quad \forall R > R_0 + 1 \, .
\end{equation}
On the other hand, the comparison principle for problem \eqref{n9} ensures that $ 0 \le \xi_n \le \| \omega \|_\infty $ in $ B_R \times (0,T) $. As a consequence, thanks to \eqref{eq:green-2} and \eqref{eq:green} we infer that $ x \mapsto \lambda \, G_R(x,o) $ is a supersolution of
\begin{equation*}\label{n9-bis}
\begin{cases}
\partial_t u  + a_n \, \Delta u = 0 & \text{in } ( B_R  \setminus \overline{B}_{R_0} ) \times (0, T) \, , \\
u = 0 & \text{on } S_R \times (0, T) \, , \\
u  = \xi_n & \text{on } S_{R_0} \times (0,T)  \, ,\\
u = 0 & \text{in } ( B_R  \setminus \overline{B}_{R_0} ) \times \{T\} \, ,
\end{cases}
\end{equation*}
whereas $ \xi_n $ is a solution to the same problem. Hence, still by comparison, there follows the validity of the inequality $ \xi_n(x,t) \le \lambda \, G_R(x,o) $ for all $ (x,t) \in ( {B}_R  \setminus \overline{B}_{R_0} ) \times (0,T) $. Since both $ \xi_n $ and $ G_R $ are nonnegative and vanish on $ S_R $, from such inequality we also deduce that
 \begin{equation}\label{der-norm-1}
\left| \frac{\partial \xi_n}{\partial \nu}(x,t) \right| \leq \lambda \left|\frac{\partial G_R}{\partial \nu} (x,o) \right| \quad \forall (x,t) \in S_R \times (0,T) \, .
\end{equation}
By plugging \eqref{der-norm-1} in \eqref{eq:est-comp-1} and exploiting \eqref{eq:green-1}, we therefore obtain
\begin{equation}\label{eq:est-4-bis}
J_n(R) \ge - \lambda \left( \alpha_\epsilon(R) + C \epsilon \right) ;
\end{equation}
if we let first $ n \to \infty $, then $ R \to+  \infty $ and finally $ \epsilon \to 0 $, estimate \eqref{eq:est-4-bis} yields
\begin{equation}\label{eq:est-4-ter}
\liminf_{R\to+\infty} \liminf_{n\to\infty} J_n(R) \ge 0 \, .
\end{equation}
As for the term $I_{n , \varepsilon}(R)$ defined in \eqref{n15}, the proof goes exactly as the one of Theorem 3.1 of \cite{GMP}, so we just sketch the main points. First of all, by multiplying the differential equation in \eqref{n9} by $\Delta\xi_n$ and integrating over $B_R\times (0, T)$ we get
\[
\frac 1 2\int_{B_R} \left| \nabla\xi_n(x,0) \right|^2 d\mu + \int_0^T \int_{B_R} a_n \left(\Delta \xi_n \right)^2 d\mu dt = \frac 1 2 \int_M \left| \nabla \omega \right|^2 d\mu \, ,
\]
so that
\begin{equation}\label{n48}
\left| I_{n , \varepsilon}(R) \right| \le \frac{\| u \|_\infty + \| v \|_\infty}{\sqrt{2}} \left( \int_M \left| \nabla \omega \right|^2 d\mu \right)^{\frac 12}  \left(\int_0^T \int_{B_R}\frac{(a-a_n)^2}{a_n} \, d\mu dt  \right)^{\frac 1 2} .
\end{equation}
It is easily shown that one can exhibit positive, regular approximations $ a_n $ of $ a $ in such a way that, at every fixed $ R>R_0+1 $,
$$
\lim_{n\to\infty} \int_0^T \int_{B_R}\frac{(a-a_n)^2}{a_n} \, d\mu dt = 0 \, ;
$$
the basic idea is to regularize $ a $ allowing for an $ L^2(B_R) $ error of order e.g.~$  1/n $ and then lift the corresponding regularizations by $ 1/n $. As a consequence, thanks to \eqref{eq:est-comp-1}, \eqref{eq:est-4-ter} and \eqref{n48}, we end up with
$$
\liminf_{R\to+\infty} \liminf_{n \to \infty} \liminf_{\varepsilon \downarrow 0} \left( I_{n , \varepsilon}(R) + J_{n , \varepsilon}(R) \right) \ge 0 \, ,
$$
whence, recalling \eqref{n14},
\begin{equation*}\label{eq:last-id}
\int_M [\overline{u}(x,T) - \underline{u}(x,T)] \, \omega(x) \, d\mu \ge 0 \, .
\end{equation*}
The conclusion follows given the arbitrariness of $ T $ and the test function $ \omega $.

Finally, the uniqueness statement is a trivial consequence of comparison.
\end{proof}

\begin{proof}[Proof of the elliptic (non-) uniqueness and comparison: Theorem \ref{thm-ell-ex} (ii)]
Our aim is to prove that $ \overline{v} \ge \underline{v} $: to this end, we shall exploit the validity of the comparison principle for the parabolic problem \eqref{pme} as stated in Theorem \ref{uniqueness-parabolic} and proved above. Indeed, the separable function
$$
\underline{u}(x,t):= \frac{\underline{v}(x)}{(t+1)^{\frac{1}{m-1}}}
$$
is by construction a subsolution of \eqref{pme} with initial datum $\underline{v}$. On the other hand, it is immediate to check that the function
$$
\overline{u}(x,t):= \frac{\overline{v}(x)}{t^{\frac{1}{m-1}}} \wedge \| \underline{v} \|_\infty
$$
is a bounded supersolution of \eqref{pme} with initial datum $ \| \underline{v} \|_\infty $ (recall that $ \overline{v} $ is positive by assumption). Thanks to \eqref{cond-elliptic-ss} it is plain that \eqref{cond-inf-bis} holds, and we are therefore in position to apply the above-mentioned parabolic comparison principle to deduce that
\begin{equation}\label{cmp-par}
\frac{\underline{v}(x)}{(t+1)^{\frac{1}{m-1}}} \le \frac{\overline{v}(x)}{t^{\frac{1}{m-1}}} \wedge \| \underline{v} \|_\infty \quad \text{for a.e.~} (x,t) \in M \times \mathbb{R}^+ \, .
\end{equation}
The thesis follows by letting $ t \to +\infty $ in \eqref{cmp-par}.

\noindent $\bullet$ Uniqueness of the minimal solution in the class of nonnegative, nontrivial, bounded solutions complying with \eqref{cond-elliptic} then follows by the just proved comparison result, upon noticing that as a consequence of standard strong maximum principles \emph{any} nonnegative, nontrivial, bounded solution to \eqref{eq:elliptic} is in fact positive.

\noindent $\bullet$  As for the construction of a positive, bounded solution satisfying \eqref{cond-elliptic-alpha}, let us start by showing that, thanks to \eqref{hp-sect}, equation
\begin{equation}\label{potential}
-\Delta w=1 \quad \text{in } M
\end{equation}
has a positive, bounded supersolution vanishing at infinity. First of all, we look for a supersolution of the form $ W(x) \equiv W(r):=K/r^{\mu_1-1}$ in $ M \setminus B_{r_0}$ (from here on we set $ \mu_1 \equiv \mu $ for notational simplicity), where $ K >0 $, $ r_0 \ge 1 $ are suitable constants to be chosen, and we set again $ B_r(o) \equiv B_r $ for all $ r>0 $. As in the beginning of the proof of Lemma \ref{ell-ssol}, from the curvature assumption we infer the validity of \eqref{eq:lap-comp}. By the monotonicity of $ r \mapsto W(r)$, it is therefore enough to require that
\[
- W''(r) - \beta r^{\mu} \, W'(r) \ge 1  \quad \forall r > r_0 \, ,
\]
and an explicit elementary computation shows that this can be done up to picking $K$ and $r_0$ large enough. In order to make it a barrier in the whole of $M$, we extend it in a $C^1$ fashion as follows (we keep denoting by $W$ such an extension):
\[
W(r):=
\begin{cases}
\frac{K}{r_0^{\mu-1}} \big[ \mu - (\mu-1) \frac{r}{r_0} \big] \ \ \ &\text{if}\ r \in [0, r_0) \, , \\
\frac{K}{r^{\mu-1}}\ \ \ &\text{if}\ r \ge r_0 \, .
\end{cases}
\]
Recalling the left-hand inequality of \eqref{eq:lap-comp}, it is apparent that $W$ becomes a \emph{weak} supersolution of \eqref{potential} provided $ K $ is sufficiently large. We can now construct, for a given $\alpha>0$, a supersolution $V$ of \eqref{eq:elliptic} which satisfies the additional condition $ V \ge \alpha$ in $M$. Indeed, by means of a direct calculation one checks that $ V $ can be chosen in the following way:
\begin{equation*}
V(x) \equiv V(r) := \left[ C \, W(r) + \alpha^m \right]^{\frac1m} ,
\end{equation*}
for any $C>0$ fulfilling
$$
C \ge \frac{1}{m-1} \left[ C \, W(0) + \alpha^m \right]^{\frac1m} .
$$
To provide the claimed solution $v_\alpha$ we use the same strategy as in \cite[Remark 4]{BK}, namely we introduce the (positive) solutions $ v_{\alpha,R} $ to the Dirichlet problems
\[
\begin{cases}
-\Delta v_{\alpha,R}^m = \frac1{m-1} \, v_{\alpha,R}  & \text{in}\ B_R \, , \\
v_{\alpha,R}=\alpha  & \text{on}\ \partial B_R \, .
\end{cases}
\]
The comparison principle in $ B_R $ ensures that $ R \mapsto v_{\alpha,R} $ is monotonically increasing and $ \alpha \le  v_{\alpha,R} \le V $ in $ B_R $. The conclusion follows by letting $ v_\alpha:=\lim_{R \to +\infty} v_{\alpha,R} $, noting that $ V $ tends to $ \alpha $ at infinity. Uniqueness is clear from \eqref{cond-elliptic-ss} and the comparison principle established in the beginning.
\end{proof}

\subsection{Asymptotics and pointwise estimates for solutions to the PME}\label{asympt}

In the first part of this section we prove Theorem \ref{thm:conv}, namely the result relating the asymptotics of  minimal  solutions to \eqref{pme} to the separable solution constructed in terms of the minimal solution to the nonlinear elliptic equation \eqref{eq:elliptic}.

The method of proof is by now classical, so we shall only stress the main points. As references, see \cite[Theorem 1.1]{V04} or \cite[Theorem 20.1]{V07} in the framework of Euclidean bounded domains, \cite[Theorem 5.1]{KRV10} in the case of a weighted Euclidean PME in the whole space and \cite[Theorem 3.1]{GMPfrac} for its fractional (nonlocal) counterpart.

\begin{proof}[Proof of Theorem \ref{thm:conv}]
First of all, the validity of \eqref{universal} is a straightforward consequence of the fact that $ v $ is strictly positive. Indeed, recalling the definition of minimal solution through the approximate solutions $ u_R $ of \eqref{pme-balls}, there exists $ t_0=t_0(R)>0 $ such that
\begin{equation*}\label{absb1}
u_0(x) \le \frac{v(x)}{ t_0^{\frac{1}{m-1}}} \quad \text{for a.e. } x \in B_R(o) \, ;
\end{equation*}
hence, as a consequence of the comparison principle on balls, there follows
\begin{equation*}\label{absb2}
u_R(x,t) \le \frac{v(x)}{\left( t + t_0 \right)^{\frac{1}{m-1}}} \quad \text{for a.e. } (x,t) \in B_R(o) \times \mathbb{R}^+
\end{equation*}
and so \eqref{universal} upon letting $ R \to \infty $. The absolute bound \eqref{universal1} is then guaranteed by the boundedness of $ v $ (Theorem \ref{thm-ell-ex} (i)).

In order to prove \eqref{convergence}, let us introduce the following \emph{rescaled} solution $U$:
\begin{equation}\label{eq: def-v}
u(x,t) =: e^{-\frac{\tau}{m-1}} \, U( x , \tau) \, , \quad t =: e^{\tau} \qquad \forall (x,\tau) \in M \times [0,+\infty) \, .
\end{equation}
It is immediate to check that $U$ is a (very weak) nonnegative solution of the equation
\[
U_\tau = \Delta U^m + \frac{1}{m-1} \, U \quad \text{in } M \times(0,+\infty) \, ,
\]
in the sense that
\begin{equation}\label{e81}
\begin{aligned}
- \int_0^{+\infty} \int_M U \, \varphi_\tau \, d\mu d\tau = & \int_0^{+\infty} \int_M U^m \, \Delta \varphi \, d\mu d\tau + \frac{1}{m-1} \int_0^{+\infty} \int_{M}  U \, \varphi \, d\mu d\tau
\end{aligned}
\end{equation}
for all $\varphi \in C^\infty_c(\mathbb{R}^d \times (0,+\infty) )$. Moreover, thanks to \eqref{universal}--\eqref{universal1},
\begin{equation}\label{e78}
U(x,\tau) \leq v(x) \leq C \quad \text{for a.e.\ } (x,\tau) \in M \times (0,+\infty) \, .
\end{equation}
On the other hand, by the B\'enilan-Crandall inequality, we know that
\begin{equation}\label{ebc}
 u_t \ge - \frac{u}{(m-1) \, t} \quad \text{in } M \times (0,+\infty) \, ,
\end{equation}
in the sense of distributions. This is a remarkable inequality that was first proved in \cite{BC} in Euclidean frameworks, but in fact only depends on the time-scaling properties of the PME (see e.g.~\cite[p.~182]{V07}), so it also holds on manifolds. Note that \eqref{ebc} can be rewritten equivalently as
\begin{equation}\label{e79}
U_\tau \geq 0 \quad \text{in } M \times (-\infty,+\infty) \, .
\end{equation}
In particular, from \eqref{e79} we deduce that $ \tau \mapsto U(\cdot,\tau) $ is (essentially) monotone increasing: therefore, the fact that $ u_0 \not \equiv 0$ and \eqref{e78} yield the existence of the (essential) monotone limit $ \hat{v}(x):= \lim_{\tau \to +\infty} U(x,\tau) $, which satisfies
\begin{equation}\label{e79-bis}
0 \le \hat{v} \le v \, , \quad \hat{v} \not\equiv 0 \, .
\end{equation}
Let us show that $ \hat{v} $ is a solution of \eqref{eq:elliptic}. By means of a standard in-time truncation argument, from \eqref{e81} it is direct to deduce the validity of
\begin{equation}\label{e81-bis}
\begin{aligned}
\int_M  U(x,\tau_2) \, \phi(x) \, d\mu - \int_M  U(x,\tau_1) \, \phi(x) \, d\mu  =  & \int_{\tau_1}^{\tau_2} \int_M U^m \, \Delta \phi \, d\mu d\tau \\
& + \frac{1}{m-1} \int_{\tau_1}^{\tau_2} \int_{M}  U \, \phi \, d\mu d\tau
\end{aligned}
\end{equation}
for all $ \phi \in C^\infty_c(M) $ and a.e.~$ \tau_2>\tau_1>0 $. By setting $ \tau_2=\tau_1+1+o(1) $ and letting $ \tau_1 \to +\infty $ in \eqref{e81-bis}, we end up with the identity
\begin{equation*}\label{e82}
0 = \int_M \hat{v}^m \, \Delta \phi \, d\mu + \frac{1}{m-1} \int_{M} \hat{v} \, \phi \, d\mu
\end{equation*}
for all $ \phi \in C^\infty_c(M) $, namely $ \hat{v} $ is a solution to \eqref{eq:elliptic}. Thanks to \eqref{e79-bis} and the definition of minimal solution, it follows that $ \hat{v} $ necessarily coincides with $ v $.

We have therefore proved that $ U(\cdot,\tau) $ converges pointwise to $ v $ as $ \tau \to +\infty $. The fact that such convergence occurs \emph{locally uniformly} just follows by the H\"{o}lder-regularity properties of bounded solutions to the PME on bounded, regular domains of $ M $: see e.g.~\cite[Section 7.6]{V07} and references therein (methods can be adapted straightforwardly on Riemannian manifolds). In order to apply these regularity results to $ U $, one can perform a direct time-scaling argument as in the proof of \cite[Theorem 20.1, Step 7]{V07}. We skip details.

In conclusion, the family $ \{ U(\cdot,\tau) \}_{\tau \ge 0} $ is locally uniformly H\"older continuous, so it converges to $ v $ as $ \tau \to +\infty $ locally uniformly by the Ascoli-Arzel\`a Theorem. The convergence then becomes \emph{global} by virtue of the decaying bounds \eqref{e78} and \eqref{eq:elliptic:above}. Finally, one retrieves the asymptotic result in terms of $ u(x,t) $ by undoing the change of variables \eqref{eq: def-v}.
\end{proof}

We now turn to the pointwise upper and lower bounds for compactly-supported solutions given in Theorems \ref{thm:main} and \ref{thm:main:low }. Since computations similar to those expanded in \cite[Proofs of Theorems 3.1 and 3.2]{GMV} are involved, we shall be fairly concise.

\begin{proof}[Proof of Theorem \ref{thm:main} (upper estimates)]
For the sake of better readability, we set again $ \mu_1 \equiv \mu $. If the upper curvature bound \eqref{hp-sect} holds, thanks to Lemma \ref{lem: ode-comparison} and Laplacian comparison as recalled in Section \ref{notation}, by reasoning as in the proof of Lemma \ref{ell-ssol}, in the region $ ( M \setminus B_1(o)) \times \mathbb{R}^+ $ it is enough to construct a radially-decreasing supersolution of the equation
\begin{equation*}\label{eq: model-upper}
u_t=\left(u^m\right)_{rr}+\beta \, r^\mu \left(u^m\right)_r \quad \text{in } [1,+\infty) \times \mathbb{R}^+ \, ,
\end{equation*}
for a suitable constant $ \beta=\beta(n, \mu_1, Q_1, R)>0 $ whose precise value here is immaterial. Indeed, the latter will also be a supersolution of $ u_t = \Delta u^m $ in $ (M \setminus B_1(o)) \times \mathbb{R}^+ $. To this purpose, let us consider the following function, which is compactly supported in space for all $t \ge 0$:
\begin{equation}\label{upper}
\overline{u}(r,t):= \frac{C}{(t+t_0)^{\frac1{m-1}}} \left[ \frac1{(r+r_0)^{\mu-1}} - \frac{\gamma}{[\log(t+t_0)]^{\frac{\mu-1}{\mu+1}}} \right]^{\frac1{m-1}}_+ \quad \forall (r,t) \in [0,+\infty) \times \mathbb{R}^+  \, ,
\end{equation}
where $ C, \gamma, r_0, t_0$ are positive parameters to be chosen later (with $ t_0>1 $). Elementary calculations, similar to those carried out in the proofs of Lemmas \ref{ell-ssol} and \ref{lemma-lower}, show that for $ \overline{u} $ to be a supersolution in the above region one must require that, for all $ r \ge 1 $ and $ t \ge 0 $, there holds
\begin{equation}\label{eq:upper-1}
\begin{aligned}
&-\frac C{m-1}\left[ \frac1{(r+r_0)^{\mu-1}}-\frac{\gamma}{[\log(t+t_0)]^{\frac{\mu-1}{\mu+1}}} \right]^{\frac1{m-1}}_+ \\
& + C \, \frac{\gamma\,(\mu-1)\,[\log(t+t_0)]^{-\frac{2\mu}{\mu+1}}}
{(\mu+1)\,(m-1)}\left[\frac1{(r+r_0)^{\mu-1}}-\frac{\gamma}{[\log(t+t_0)]^{\frac{\mu-1}{\mu+1}}}\right]^{\frac{2-m}{m-1}}_+\\
\ge & \, C^m \, \frac{m\,\mu\,(\mu-1)}{m-1}\,(r+r_0)^{-\mu-1}\left[\frac1{(r+r_0)^{\mu-1}} - \frac{\gamma}{[\log(t+t_0)]^{\frac{\mu-1}{\mu+1}}}\right]^{\frac1{m-1}}_+\\
& + C^m \, \frac{m\,(\mu-1)^2}{(m-1)^2} \, (r+r_0)^{-2\mu}\left[\frac1{(r+r_0)^{\mu-1}}-\frac{\gamma}{[\log(t+t_0)]^{\frac{\mu-1}{\mu+1}}}\right]^{\frac{2-m}{m-1}}_+\\
& - C^m \, \frac{\beta \, (\mu-1)}{m-1} \, r^{\mu} \, (r+r_0)^{-\mu} \left[\frac1{(r+r_0)^{\mu-1}}-\frac{\gamma}{[\log(t+t_0)]^{\frac{\mu-1}{\mu+1}}}\right]^{\frac1{m-1}}_+ .
\end{aligned}
\end{equation}
It is not difficult to check that, if $ r_0 \ge \overline{r}_0 $ and $ C \ge \overline{C}_0 $ for suitable $ \overline{r}_0,\overline{C}_0>0$ depending on $n,m,\mu,\beta$ and on $n,m,\mu,\beta, r_0 $, respectively, inequality \eqref{eq:upper-1} is in fact equivalent to
\begin{equation}\label{eq:upper-2}
\begin{aligned}
& C^{m-1} \, k_1 \left[ \frac1{(r+r_0)^{\mu-1}}-\frac{\gamma}{[\log(t+t_0)]^{\frac{\mu-1}{\mu+1}}}\right]^{\frac1{m-1}}_+ \\
& + \gamma \, k_2 \, [\log(t+t_0)]^{-\frac{2\mu}{\mu+1}}\left[ \frac1{(r+r_0)^{\mu-1}}-\frac{\gamma}{[\log(t+t_0)]^{\frac{\mu-1}{\mu+1}}}\right]^{\frac{2-m}{m-1}}_+ \\
\ge & \, C^{m-1} \, (r+r_0)^{-2\mu} \left[ \frac1{(r+r_0)^{\mu-1}}-\frac{\gamma}{[\log(t+t_0)]^{\frac{\mu-1}{\mu+1}}}\right]^{\frac{2-m}{m-1}}_+ ,
\end{aligned}
\end{equation}
where $ k_1 , k_2>0 $ are other numerical constants that we need not make explicit: we just point out that $ k_2 $ depends only on $n,m,\mu,\beta$, whereas $ k_1>0 $ depends on $n,m,\mu,\beta$ and continuously on $ r_0 $, behaving like $ r_0^{-\mu} $ as $ r_0 $ grows. It suffices to consider inequality \eqref{eq:upper-2} in the space-time region
\begin{equation}\label{eq:upper-region}
\left\{ (r,t) \in [1,+\infty)\times \mathbb{R}^+: \ r+r_0 < \gamma^{-\frac1{\mu-1}} \, [\log(t+t_0)]^{\frac1{\mu+1}} \right\} ,
\end{equation}
i.e.~where the quantity appearing in the positive-part brackets is indeed positive (outside the region \eqref{eq:upper-region} the differential inequality is trivially satisfied in a weak sense). There \eqref{eq:upper-2} amounts to
\begin{equation*}\label{eq:upper-3}
C^{m-1} \, k_1 \left[ \frac1{(r+r_0)^{\mu-1}}-\frac{\gamma}{[\log(t+t_0)]^{\frac{\mu-1}{\mu+1}}}\right]+\frac{\gamma \, k_2}{[\log(t+t_0)]^{\frac{2\mu}{\mu+1}}} \ge \frac{C^{m-1}}{(r+r_0)^{2\mu}} \, ,
\end{equation*}
or equivalently
\begin{equation}\label{eq:upper-4}
C^{m-1} \left[\frac{k_1}{(r+r_0)^{\mu-1}} - \frac{1}{(r+r_0)^{2\mu}} \right]- \frac{C^{m-1} \, \gamma\,k_1}{[\log(t+t_0)]^{\frac{\mu-1}{\mu+1}}} +\frac{\gamma \, k_2}{[\log(t+t_0)]^{\frac{2\mu}{\mu+1}}} \ge 0 \, .
\end{equation}
It is readily seen that, for instance, if~$ r_0^{\mu+1} \ge 2\mu/[k_1(\mu-1)] $ (which is always possible provided $\overline{r}_0$ is large -- recall that $ k_1 \sim r_0^{-\mu} $) then the l.h.s.~of \eqref{eq:upper-4} is decreasing as a function of $r$, so that the validity of \eqref{eq:upper-4} in the region \eqref{eq:upper-region} is equivalent to the validity of the latter at its (spatial) right-hand extremum, which in turn reads
\begin{equation}\label{eq:upper-5}
\gamma^{\frac{\mu+1}{\mu-1}} \le \frac{k_2}{C^{m-1}} \, .
\end{equation}
We have therefore established that under the conditions $ r_0 \ge \overline{r}_0$, $ C \ge \overline{C}_0$ and \eqref{eq:upper-5} the function \eqref{upper} is indeed a supersolution in $ (M \setminus {B_1}(o)) \times \mathbb{R}^+ $. Note that $ t_0 > 1 $ for the moment is left as a free parameter.

In order to make $ \overline{u} $ a \emph{global} supersolution, we still have to deal with the region $ B_1(o) \times \mathbb{R}^+ $. Thanks to Theorem \ref{thm:conv}, we know in particular that there exists $ \tau,\mathcal{M}>0 $, depending on $ v$ and on $ \| u_0 \|_\infty $, such that
\begin{equation*}\label{eq:upper-6}
u(x,t) \le \frac{\mathcal{M}}{(t+\tau)^{\frac{1}{m-1}}} \quad \text{for a.e. } (x,t) \in B_1(o) \times \mathbb{R}^+ \, .
\end{equation*}
Some elementary calculations show that $ \overline{u} \ge {\mathcal{M}}/{(t+\tau)^{\frac{1}{m-1}}} $ in $ B_1(o) \times \mathbb{R}^+ $ if e.g.
\begin{equation}\label{eq:upper-7}
\gamma \le \frac{(\log t_0)^{\frac{\mu-1}{\mu+1}}}{2\,(1+r_0)^{\mu-1}} \, , \quad t_0 \ge \tau \, , \quad C^{m-1} \ge 2 \, (1+r_0)^{\mu-1} \, \frac{t_0}{\tau} \, \mathcal{M}^{m-1}  \, .
\end{equation}
Finally, in addition to $ r_0 \ge \overline{r}_0 $, $ C \ge \overline{C}_0 $, \eqref{eq:upper-5} and \eqref{eq:upper-7}, we have to impose that $ \overline{u}(x,0) \ge u(x,0) $ for a.e.~$ x \in M \setminus B_1(o) $, which is easily achieved upon choosing $ \gamma $ small enough, to make sure that the support of $ \overline{u}(x,0) $ covers the one of $ u(x,0) $, and $ C $ large enough, to make sure that the minimum of $ \overline{u}(x,0) $ restricted to the support of $ u(x,0) $ is larger than the (essential) maximum of $ u(x,0) $. Upon labeling any of such choices by $ C=C_1 $, $ \gamma=\gamma_1 $, $ r_0=r_1 $ and $ t_0=t_1 $, we end up with \eqref{upper1}.
\end{proof}

\begin{proof}[Proof of Theorem \ref{thm:main:low } (lower estimates)]
Again, in order to make notations more readable, we set $ \mu_2 \equiv \mu $. If the lower curvature bound \eqref{hp-ricc} holds, thanks to Lemma \ref{lem: ode-comparison-2} and Laplacian comparison, by reasoning as in the proof of Lemma \ref{lemma-lower}, in the region $ ( M \setminus B_1(o)) \times \mathbb{R}^+ $ it is enough to find a radially-decreasing subsolution of the equation
\begin{equation*}\label{eq: model-upper1}
u_t=\left(u^m\right)_{rr} + \hat{\beta} \, r^\mu \left(u^m\right)_r \quad \text{in } [1,+\infty) \times \mathbb{R}^+ \, ,
\end{equation*}
for a suitable $ \hat \beta = \hat \beta(n,\mu_2,Q_2,R,\inf_{x \in B_R(o)}  \mathrm{Ric}_o(x))>0 $ whose precise value is unnecessary to our purposes. The latter will also be a subsolution of $ u_t = \Delta u^m $ on $ (M \setminus B_1(o)) \times \mathbb{R}^+ $. We consider a subsolution $ \underline{u}(r,t) $ of the same type as \eqref{upper}: hence, for all $ r\ge 1 $ and $ t \ge 0 $ we need to require the validity of the analogue of \eqref{eq:upper-1} with reverse inequality. By arguing similarly to above, one sees that if e.g.~$ r_0 \ge 1 $ and $ 0 <  C \le \underline{C}_0 $ (for a suitable $ \underline{C}_0>0$ depending only on $n,m,\mu,\hat{\beta}$) such inequality turns out to be equivalent, for all $ r \ge 1 $ and $ t \ge 0 $, to
\begin{equation}\label{eq:lower-1}
\begin{aligned}
& -\left[ \frac1{(r+r_0)^{\mu-1}}-\frac{\gamma}{[\log(t+t_0)]^{\frac{\mu-1}{\mu+1}}}\right]^{\frac1{m-1}}_+ \\
& + \gamma \, k_1 \, [\log(t+t_0)]^{-\frac{2\mu}{\mu+1}}\left[ \frac1{(r+r_0)^{\mu-1}}-\frac{\gamma}{[\log(t+t_0)]^{\frac{\mu-1}{\mu+1}}}\right]^{\frac{2-m}{m-1}}_+ \\
\le & \, C^{m-1} \, k_2 \, (r+r_0)^{-2\mu}\left[\frac1{(r+r_0)^{\mu-1}}-\frac{\gamma}{[\log(t+t_0)]^{\frac{\mu-1}{\mu+1}}}\right]^{\frac{2-m}{m-1}}_+ ,
\end{aligned}
\end{equation}
where $k_1,k_2 > 0 $ are other suitable constants depending only on $ n,m,\mu,\hat\beta $. Again, it suffices to make sure that \eqref{eq:lower-1} holds in the region \eqref{eq:upper-region}, which leads to
\begin{equation}\label{eq:lower-2}
\left[ -\frac1{(r+r_0)^{\mu-1}}+\frac{\gamma}{[\log(t+t_0)]^{\frac{\mu-1}{\mu+1}}} \right] + \frac{\gamma \, k_1}{[\log(t+t_0)]^{\frac{2\mu}{\mu+1}}} \le \frac{C^{m-1} \, k_2}{(r+r_0)^{2\mu}} \, .
\end{equation}
By monotonicity, it is readily seen that the worst case occurs at the (spatial) right-hand extremum of \eqref{eq:upper-region}, where \eqref{eq:lower-2} amounts to
\begin{equation}\label{eq:lower-3}
\gamma^{\frac{\mu+1}{\mu-1}} \ge \frac{k_1}{C^{m-1} \, k_2} \, .
\end{equation}
Hence, under the conditions $ r_0 \ge 1$, $ 0< C \le \underline{C}_0$ and \eqref{eq:lower-3}, the function \eqref{upper} is a subsolution in $ ( M \setminus B_1(o) ) \times \mathbb{R}^+ $. We remark that $ t_0 > 1 $ is still left as a free parameter.

Let us deal with inner barriers. As observed in \cite[Proof of Theorem 3.2]{GMV}, due to the diffusive nature of the PME it is plain that there exists $ T>0 $, depending on $ u_0 $, $ m $ and the metric of $ M $ (only locally), such that
\begin{equation*}\label{eq:lower-final-1}
\inf_{x \in {B}_2(o)} u(x,T)  > 0 \, .
\end{equation*}
Therefore, a standard separable-subsolution argument on $ B_1(o) $ ensures that there exists a suitable constant $ \mathcal{I}>0 $ (depending on the just mentioned quantities) complying with
\begin{equation*}\label{eq:lower-final-2}
u(x,t) \ge \frac{\mathcal{I}}{t^{\frac{1}{m-1}}} \quad \text{for a.e. } (x,t) \in {B}_1(o) \times (T,+\infty) \, .
\end{equation*}
Finally, it is straightforward to check that the further conditions
\begin{equation}\label{eq:lower-final-3}
C \le \mathcal{I} \, r_0^{\frac{\mu-1}{m-1}} \, , \quad \gamma \ge \frac{[\log(T+t_0)]^{\frac{\mu-1}{\mu+1}}}{r_0^{\mu-1}} \, ,
\end{equation}
yield $ \underline{u}(x,t) \equiv 0 $ for all $ (x,t) \in (M \setminus B_1(o)) \times (0,T) $ and $ \underline{u}(x,t) \le \mathcal{I}/t^{\frac{1}{m-1}} $ for all $ (x,t) \in \partial B_1(o) \times [T,+\infty) $, which is enough to assert the validity of \eqref{lower1}, upon labeling by $ C=C_2 $, $ \gamma=\gamma_2 $, $ r_0=r_2 $ and $ t_0=t_2 $ any of the above choices that satisfy $ 0 < C \le \underline{C}_0 $, $ r_0 \ge 1 $, \eqref{eq:lower-3} and \eqref{eq:lower-final-3}.
\end{proof}

\section{Application to weighted Euclidean porous medium equations}\label{weighted}
In this section we state and briefly prove asymptotic results for nonnegative solutions to the following \emph{weighted Euclidean} PME posed in the whole space:
\begin{equation}\label{weighted-pme}
\begin{cases}
\rho \, u_t=\Delta u^m  & \text{in } \mathbb{R}^n \times \mathbb{R}^+ \, , \\
u(\cdot,0)=u_0  & \text{in } \mathbb{R}^n  \, ,
\end{cases}
\end{equation}
with a new class of decaying weights $\rho$ that had not been treated before. We require that $n\ge3$ for technical reasons. In fact, the barriers we use in this section do not work in the case $n=2$. We proceed again by means of barrier arguments, where the choice of the explicit barriers is strictly related to results on manifolds proved in this paper, through a change of (radial) variable first introduced in \cite{V} and then exploited in greater generality in \cite{GMV}. The curvature assumptions required here translate, according to the just-mentioned change of variables, into precise assumptions on the asymptotic behaviour of the weight. More precisely, this leads us to deal with weights having a \emph{supercritical} decay (though near critical), in the sense that
$$ \rho(x) \sim |x|^{-2 }\, (\log |x|)^{-\nu} \quad \text{ as } |x|=\operatorname{d}(x,0)\to+\infty \, , \quad \text{with} \ \nu>1 \,, $$
a range that had not been considered in the study of \emph{subcritical} weights carried out in \cite[Section 9]{GMV} (see Theorems 9.1 and 9.2 there).

\begin{thm}\label{weight}
Let $ \rho(x) $ be a positive function such that $ \rho \in L^\infty(\mathbb{R}^n) $ ($ n \ge 3 $), $ \rho^{-1} \in L^\infty(\mathbb{R}^n)$ and
\begin{equation}\label{rho.ass}
\frac{K_1}{|x|^2 \, (\log |x|)^{\nu}} \le  \rho(x) \le \frac{K_2}{|x|^2 \, (\log |x|)^{\nu}} \quad \text{for a.e. } x \in \mathbb{R}^n \setminus B_2(0) \, ,
\end{equation}
for some $ \nu>1 $ and $ K_1,K_2>0 $. Let $u$ be the minimal solution to \eqref{weighted-pme} corresponding to a nonnegative, nontrivial and compactly supported initial datum $ u_0 \in L^\infty(M) $. Then:

\noindent {(i)} The pointwise estimates
\begin{equation}\label{weighted-pme1}
\begin{aligned}
&\frac{C_1}{(t+t_0)^{\frac{1}{m-1}}} \left[ \frac{1}{\left( \log |x| \right)^{\nu-1}} - \frac{\gamma_1}{\left[\log (t+t_0) \right]^{\nu-1}} \right]_+^{\frac{1}{m-1}}
\le
u(x,t)
\le \\
& \frac{C_2}{(t+t_0)^{\frac{1}{m-1}}} \left[ \frac{1}{\left( \log |x| \right)^{\nu-1}} - \frac{\gamma_2}{\left[\log (t+t_0) \right]^{\nu-1}} \right]_+^{\frac{1}{m-1}}
\end{aligned}
\end{equation}
hold for a.e.~$ x \in \mathbb{R}^n \setminus B_{R_0}(0) $ and $ t \ge 0 $, where $ C_1,C_2,\gamma_1, \gamma_2,R_0,t_0 $ are suitable positive constants depending on $ n,m,\nu,\rho,u_0$.

\noindent {(ii)} Furthermore,
\begin{equation}\label{sepbeh.type}
\lim_{t\to+\infty} \left\| t^{\frac{1}{m-1}} u(\cdot,t) - V  \right\|_{L^\infty(\mathbb{R}^n)} = 0 \, ,
\end{equation}
where $V$ is the minimal, positive solution to the equation
\begin{equation}\label{ellweight}
-\Delta V^m = \frac{\rho(x)}{m-1} \, V \quad \text{in } \mathbb{R}^n
\end{equation}
as constructed in \cite{BK}, which satisfies the asymptotic estimates
\begin{equation}\label{asymptweight}
\begin{aligned}
\left[ \frac{K_1}{m(\nu-1)(n-2)} \right]^{\frac{1}{m-1}} \le & \liminf_{|x|\to+\infty} V(x) \left( \log |x| \right)^{\frac{\nu-1}{m-1}} \\
\le & \limsup_{|x|\to+\infty} V(x) \left( \log |x| \right)^{\frac{\nu-1}{m-1}} \le \left[ \frac{K_2}{m(\nu-1)(n-2)} \right]^{\frac{1}{m-1}} .
\end{aligned}
\end{equation}
\end{thm}
The above theorem is to be compared to the results of \cite{KRV10} for weights with decay $\rho(x)\sim |x|^{-\gamma}$ for some $\gamma $ mildly larger than 2. So, by combining the results of \cite[Section 9]{GMV} and Theorem \ref{weight}, we can conclude that the behaviour
$$
\rho(x) \sim |x|^{-2}\,(\log|x|)^{-1} \quad \text{ as } |x| \to + \infty
$$
is critical with respect to the appearance of a separable asymptotics as in \eqref{sepbeh.type}.

\begin{proof}[Proof of Theorem \ref{weight}]
As concerns the bound \eqref{weighted-pme1}, first of all let us show that the function
\begin{equation}\label{barrier2}
\overline{u}(x,t):=\frac{C_2}{(t+t_0)^{\frac{1}{m-1}}} \left[ \frac{1}{\left( \log |x| \right)^{\nu-1}} - \frac{\gamma_2}{\left[\log (t+t_0) \right]^{\nu-1}} \right]_+^{\frac{1}{m-1}}
\end{equation}
is a supersolution of equation \eqref{weighted-pme}. In order to improve readability, we set
\[
\mathcal A(x,t):=\left[ \frac{1}{\left( \log |x| \right)^{\nu-1}} - \frac{\gamma_2}{\left[\log (t+t_0) \right]^{\nu-1}} \right]_+ .
\]
By direct calculations one sees that $\overline{u}$ is a supersolution in $ (\mathbb{R}^n \setminus B_{R_0}(0)) \times \mathbb{R}^+ $ if and only if, in such region, there holds
\begin{equation}\label{eq: weight-barr}
\begin{aligned}
& -\frac{K_2 \, C_2}{m-1} \, \frac{\mathcal A(x,t)^{\frac{1}{m-1}}}{|x|^2 \, (\log |x|)^\nu} + \frac{K_2 \, \gamma_2 \, C_1 \, (\nu-1)}{m-1}\,\frac{\mathcal A(x,t)^{\frac{2-m}{m-1}}}{|x|^2 \, (\log |x|)^\nu \, [\log (t+t_0)]^\nu} \\
\ge & -\frac{C_2^m \, m \, (\nu-1)}{m-1}\left(n-2-\frac{\nu}{\log|x|}\right)\frac{\mathcal A(x,t)^{\frac{1}{m-1}}}{|x|^2 \, (\log |x|)^\nu}+\frac{C_2^m \, m \,(\nu-1)^2}{(m-1)^2} \, \frac{\mathcal A(x,t)^{\frac{2-m}{m-1}}}{|x|^2 \, (\log|x|)^{2\nu}} \, ,
\end{aligned}
\end{equation}
where $ R_0>2 $ and $ t_0>1 $ with no loss of generality. By close analogy with the explicit computations expanded in the proof of Theorem \ref{thm:main}, it is possible to show that \eqref{eq: weight-barr} is satisfied in the required space-time range provided $C_2,R_0,t_0$ are large enough and $\gamma_2$ is small enough. A very similar calculation shows that the function
\begin{equation}\label{eq:lowb}
\underline{u}(x,t):=\frac{C_1}{(t+t_0)^{\frac{1}{m-1}}} \left[ \frac{1}{\left( \log |x| \right)^{\nu-1}} - \frac{\gamma_1}{\left[\log (t+t_0) \right]^{\nu-1}} \right]_+^{\frac{1}{m-1}}
\end{equation}
is a subsolution of \eqref{weighted-pme} provided $t_0$ is the one chosen above, $ C_1 $ is small enough and $\gamma_1,R_0$ are large enough. Such barriers can be fixed inside a ball by reasoning likewise the final parts of the proofs of Theorems \ref{thm:main} and \ref{thm:main:low }. Parameters can also be selected so that the resulting barriers are above and below $u_0$, respectively. As for \eqref{sepbeh.type}, one proceeds as in the proof of Theorem \ref{thm:conv} provided existence of a positive (minimal) solution to the nonlinear elliptic equation \eqref{ellweight} is guaranteed. The latter fact follows from a direct application of the results of \cite{BK}, or by barrier arguments of the same type as those used in the proof of Theorem \ref{thm-ell-ex}. Indeed, the analogues of such barrier arguments yield precisely the asymptotic estimates \eqref{asymptweight}.
\end{proof}

\begin{oss} \rm The explicit form of the sub- and supersolutions exploited in the previous theorem is strictly related to the analogous results on manifolds, namely Theorems \ref{thm:main} and \ref{thm:main:low }. Indeed, let us briefly recall the change of variables introduced in \cite{V, GMV}. Consider a model manifold $M_\psi $ as in Section \ref{notation}, on which the PME \eqref{pme} reads
\begin{equation}\label{eq.hpme}
u_t=\Delta_g\!\left(u^m\right) \quad \text{with} \quad \Delta_{g} f = f^{\prime\prime} + (n-1) \, \frac{\psi^\prime}{\psi} \, f^\prime
\end{equation}
for all regular radial functions $ f=f(r) $, where the metric $ g $ on $ M_\psi $ is completely determined by $ \psi $. When looking for \emph{radial} solutions to \eqref{eq.hpme}, through a suitable change of variables $s=s(r)$ one can transform \eqref{eq.hpme} into an equation for ${\hat u}(s,t)= u(r(s),t)$ of the form
\begin{equation*}\label{eq.wpme}
\rho(s) \, \hat{u}_t= \Delta_s \! \left( \hat{u}^m \right) ,
\end{equation*}
where now $\Delta_s$ is the \emph{Euclidean} radial Laplacian in dimension $n$ (the radial coordinate being $s\equiv|x|$)
and $ \rho(s) $ is a suitable positive weight. The required condition for the change of variables to work as above turns out to be
\begin{equation}\label{eq.diff}
\frac{{d}s}{s^{n-1}}=\frac{{d}r}{\psi(r)^{n-1}} \, .
\end{equation}
Upon integrating \eqref{eq.diff} between $r$ and $+\infty$ and using the asymptotic properties of $\psi$ that follow from our curvature bounds, one can show (see \cite{GMV} for details) that $\rho$ is implicitly given by
\begin{equation*}\label{rho.def}
\rho(s)=\frac{\left[\psi\!\left(r(s)\right) \right]^{2(n-1)}}{s^{2(n-1)}} \quad \forall s \in (0,+\infty) \, ,
\end{equation*}
that it is regular, positive and complies with the bounds \eqref{rho.ass}. Hence, in principle under additional assumptions on $\rho$, one can translate the results we proved on manifolds to the weighted Euclidean setting. However, a posteriori, barriers as in \eqref{barrier2} and \eqref{eq:lowb} yield super- and subsolutions of \eqref{weighted-pme} under the sole assumption \eqref{rho.ass}.
\end{oss}


\section{Comments and extensions}

\noindent $\bullet$  We have extended our previous investigations on asymptotic behaviour of the Porous Medium Equation on Riemannian manifolds to the case of very curved Cartan-Hadamard manifolds with curvature exponents $2\mu>2$. We note that the bulk behaviour is quite new, in the form of a separate-variable behaviour related to a special nonlinear elliptic equation, whose study has independent interest, and this is reminiscent of Dirichlet problems on bounded domains.

\noindent $\bullet$ On the other hand, the propagation properties depend on the outer behaviour and then there is  continuity for the radius (and to some extent volume) size of the support across the value $\mu=1$. Actually, the support radius behaves like $(\log t)^{1/(1+\mu)}$, the exponent going all the way from infinity at $\mu=-1$ to 0 as $ \mu \to +\infty$. Regarding the estimate of the volume of the support, at least on model manifolds we have
$$
\mathsf{V}(t) := \mathrm{Vol}(B_{\mathsf R(t)}(o))=\mathrm{C}(n)\int_0^{\mathsf R(t)} \psi(r)^{n-1}\,dr \,, \quad \frac{\psi'}{\psi}\sim r^{\mu}\,;
$$
if we knew that $\mathsf R(t) \sim (c \, \log t)^{1/(1+\mu)}$ for a precise $ c>0 $, we would get at leading order the behaviour $ \mathsf{V}(t) \sim t^{{(n-1)c}/(\mu+1)}$ as $ t \to +\infty $, with possible logarithmic corrections. A bound of this form for $\mathsf R(t)$ is natural on model manifolds satisfying our curvature conditions as equalities, in view of \eqref{bound-smooth-infty} and \eqref{bound-smooth-infty-lower}. In order to make comparisons we need an accurate value of $c$:  this was done in the case of the hyperbolic space in \cite{V} but is nontrivial in the present situation and is left as an open problem. Nonetheless, a general lower bound can be estimated by an easy argument and holds under the sole \emph{upper} curvature assumption: thus, if $B(t):=B_{\mathsf R(t)}(o)$ we get, by the mass conservation property  and \eqref{universal1},
$$
\mathsf M:=\int_{B(t)} u(x,t)\,d\mu \le \|u(t)\|_{L^\infty(M)} \, \mathsf V(t) \le C \, t^{-\frac{1}{m-1}} \, \mathsf V(t)  \, ,
$$
which yields
$$
\mathsf V(t) \ge C \, \mathsf M \, t^{\frac{1}{m-1}} \, .
$$
This is a good bound from below and it gives a lower estimate  on the above constant $c$.  In particular, volume filling occurs at least as fast as in $\mathbb{H}^n$ and faster than in $\mathbb{R}^n$. Clearly, a bound from above in terms of a different power of $t$ also holds in view of \eqref{bound-smooth-infty} under upper curvature bounds. We conjecture that the correct power is indeed $t^{1/(m-1)}$, with possible logarithmic corrections.
%

\noindent $\bullet$  We have considered a representative family of curvatures with power divergence $O(r^{2\mu})$ as $r\to\infty$ and we have found explicit support rates with exponents that go to zero as $\mu\to+\infty$. Moreover, by taking a more divergent curvature function we could have obtained support propagation rates as slow as we wanted. Finally, the Dirichlet problem in a bounded domain  $\Omega$  of $M$ could be obtained as limit of the solutions in perturbations $M_k$ of $M$ (having $\Omega $ in common) that have increasing curvatures in the complement of $\Omega$. We leave these developments to the interested reader.

\noindent $\bullet$ The sublinear elliptic equation is important in itself and is worth a final comment. Brezis and Kamin in \cite{BK} show uniqueness of solutions to such equation in the weighted Euclidean case, under appropriate conditions on the weight; they consider a wider class of solutions, namely those whose $\liminf$  at infinity vanishes. However, their arguments cannot be easily adapted to the manifold framework, since they involve Liouville-type results that may fail under the present geometric assumptions on $M$, see e.g.~\cite{TA}. It is an open problem to establish whether uniqueness holds in even larger classes of solutions, possibly after requiring additional properties of the manifold.

\bigskip

\noindent {\sc Acknowledgment.} J.L.V.~was supported by Spanish Project MTM2014-52240-P. G.G.~was partially supported by the PRIN Project {``Equazioni alle derivate parziali di tipo ellittico e parabolico: aspetti geometrici, disuguaglianze collegate, e applicazioni''} (Italy). M.M.~was partially supported by the GNAMPA Project ``Equazioni diffusive non-lineari in contesti non-Euclidei e disuguaglianze funzionali associate'' (Italy). Both G.G.~and M.M.~have also been supported by the Gruppo Nazionale per l'Analisi Matematica, la Probabilit\`a e le loro Applicazioni (GNAMPA) of the Istituto Nazionale di Alta Matematica (INdAM, Italy).



\begin{thebibliography}{999}

\bibitem{TA} M. Arnaudon, A. Thalmaier, \emph{Brownian motion and negative curvature. Random walks, boundaries and spectra}, 143--161, Progr. Probab., \textbf{64}, Birkh\"{a}user/Springer Basel AG, Basel, 2011.

\bibitem{BC} P. B\'enilan, M.~G. Crandall, \emph{The continuous dependence on $\varphi$ of solutions of $u_t - \Delta \varphi(u) = 0$}, Indiana Univ. Math. J. \bf 30 \rm (1981), 161--177.

\bibitem{BCP} P. B\'enilan, M.~G. Crandall, M.~Pierre, \emph{Solutions of the porous medium equation in $\mathbb R^N$ under optimal conditions on initial values}, Indiana Univ.~Math.~J. {\bf 33} (1984), 51--87.

\bibitem{BV} M. Bonforte, J.~L. V\'azquez, {\em Global positivity estimates and Harnack inequalities for the fast diffusion equation}, J. Funct. Anal. \bf 240 \rm (2006), 399--428.

\bibitem{BK} H. Brezis, S. Kamin, {\em Sublinear elliptic equations in ${\mathbb R}^n$}, Manuscripta Math. \bf 74 \rm (1992), 87--106.

\bibitem{DKV} E. DiBenedetto, Y. Kwong, V. Vespri, {\em Local space-analyticity of solutions of certain singular parabolic equations}, Indiana Univ. Math. J. \bf 40 \rm (1991), 741--765.

\bibitem{EK} D. Eidus, S. Kamin, \emph{The filtration equation in a class of functions decreasing at infinity}, Proc. Amer. Math. Soc. \bf 120 \rm (1994), 825--830.

\bibitem{GW} R.~E. Greene, H. Wu, ``Function Theory on Manifolds Which Possess a Pole'', Lecture Notes in Mathematics, 699. Springer, Berlin, 1979.

\bibitem{Grig} A. Grigor'yan, \emph{Analytic and geometric background of recurrence and non-explosion of the Brownian motion on Riemannian
manifolds}, Bull.~Amer.~Math.~Soc. {\bf 36} (1999), 135--249\,.

\bibitem{GM3} G. Grillo, M. Muratori, {\em Smoothing effects for the porous medium equation on Cartan-Hadamard manifolds}, Nonlinear Anal. \bf 131 \rm (2016), 346--362.

\bibitem{GMPc} G. Grillo, M. Muratori, F. Punzo, \emph{Conditions at infinity for the inhomogeneous filtration equation}, Ann. Inst. H. Poincar\'e Anal. Non Lin\'eaire \textbf{31} (2014), 413--428.

\bibitem{GMPfrac} G. Grillo, M. Muratori, F. Punzo, \emph{On the asymptotic behaviour of solutions to the fractional porous medium equation with variable density}, Discrete Contin. Dyn. Syst. \textbf{35} (2015), 5927--5962.

\bibitem{GMP} G. Grillo, M. Muratori, F. Punzo, \emph{The porous medium equation with large initial data on negatively curved Riemannian manifolds}, to appear on J. Math. Pures Appl., \url{https://doi.org/10.1016/j.matpur.2017.07.021}.

\bibitem{GMV} G. Grillo, M. Muratori, J.~L. V\'azquez, {\em The porous medium equation on Riemannian manifolds with negative curvature. The large-time behaviour}, Adv. Math. \textbf{314} (2017), 328--377.

\bibitem{I} K. Ishige, \em An intrinsic metric approach to uniqueness of the positive Dirichlet problem for parabolic equations in cylinders\rm, J. Differential Equations \bf 158 \rm(1999), 251--290\,.

\bibitem{I2} K. Ishige, \em An intrinsic metric approach to uniqueness of the positive
Cauchy-Neumann problem for parabolic equations\rm , J. Math. Anal. Appl. \bf 276 \rm (2002), 763--790\,.

\bibitem{IM} K. Ishige, M. Murata, \em Uniqueness of nonnegative solutions of the Cauchy problem for parabolic equations on manifolds or domains\rm, Ann. Scuola Norm. Sup. Pisa Cl. Sci. {\bf 30} (2001), 171--223\,.

\bibitem{KP} S. Kamin, F. Punzo, \emph{Dirichlet conditions at infinity for parabolic and elliptic equations}, Nonlinear Anal. \textbf{138} (2016), 156--175.

\bibitem{KRV10} S. Kamin, G. Reyes, J.~L. V\'azquez, \emph{Long time behavior for the inhomogeneous PME in a medium with rapidly decaying density}, Discrete Contin. Dyn. Syst. \textbf{26} (2010), 521--549.

\bibitem{KT} R. Kersner, A. Tesei, \emph{Well-posedness of initial value problems for singular parabolic equations}, J. Differential Equations \textbf{199} (2004), 47--76.

\bibitem{M1} M. Murata, \emph{Nonuniqueness of the positive Dirichlet problem for parabolic equations in cylinders}, J. Funct. Anal. \bf 135 \rm (1996), 456--487\,.

\bibitem{M2} M. Murata, \emph{Heat escape}, Math. Ann. \bf 327 \rm (2003), 203--226\,.

\bibitem{Pierre} M. Pierre, \emph{Uniqueness of the solutions of $u_t - \Delta \varphi(u) = 0$ with initial datum a measure}, Nonlinear Anal. {\bf 6} (1982), 175--187.

\bibitem{Pu} F. Punzo, \emph{Uniqueness and non-uniqueness of solutions to quasilinear parabolic equations with a singular coefficient on weighted Riemannian manifolds}, Asymptot. Anal. \textbf{79} (2012), 273--301.

\bibitem{V04} J.~L. V\'azquez, \emph{The Dirichlet problem for the porous medium equation in bounded domains. Asymptotic behavior}, Monatsh.~Math. \textbf{142} (2004), 81--111.

\bibitem{V07} J.~L. V{\'a}zquez, {``The Porous Medium Equation. Mathematical Theory''}, Oxford Mathematical Monographs. The Clarendon Press, Oxford University Press, Oxford, 2007.

\bibitem{V} J.~L. V\'azquez, {\em Fundamental solution and long time behavior of the porous medium equation in hyperbolic space}, J. Math. Pures Appl. \bf 104 \rm (2015), 454--484.

\end{thebibliography}
\end{document}